\documentclass[12pt,a4paper]{amsart}

\usepackage[utf8]{inputenc}
\usepackage[english]{babel}

\usepackage{latexsym}
\usepackage{amsfonts}
\usepackage{amsmath}
\usepackage{mathtools}
\usepackage{amssymb}
\usepackage{graphicx}
\usepackage{color}
\usepackage{xcolor}
\usepackage{enumitem}
\usepackage{tikz}
\usetikzlibrary{hobby}
\usetikzlibrary{arrows.meta}

\usepackage{algorithm}
\usepackage{algcompatible}
\usepackage{fullpage}
\usepackage{url}
\usepackage[ocgcolorlinks, linkcolor=red]{hyperref}

\usepackage{csquotes} 
\usepackage[
  backend=biber,
  style=alphabetic,
  giveninits=true, 
  terseinits=true, 
  maxnames=6, 
  natbib=true, 
  sortlocale=en_GB,
  url=false,
  doi=true,
  eprint=true
  ]{biblatex}
\AtEveryBibitem{\clearfield{issn}}
\AtEveryBibitem{\clearfield{isbn}}

\newcommand{\R}{\mathbb{R}}
\newcommand{\N}{\mathbb{N}}

\newcommand{\C}{\mathbb{C}}

\newcommand{\eps}{\epsilon}

\renewcommand\Re{\operatorname{Re}}

\newcommand\p{\partial}

\newtheorem{theorem}{Theorem}[section]

\newtheorem{lemma}[theorem]{Lemma}
\newtheorem{proposition}[theorem]{Proposition}

\newtheorem*{question*}{Question}

\theoremstyle{definition}
\newtheorem{remark}[theorem]{Remark}
\newtheorem*{remark*}{Remark}

\numberwithin{equation}{section}

\begin{document}

\title[]{On the growth properties of interior transmission eigenfunctions near corners}

\date{}

\author{Emilia L.K. Bl{\aa}sten}
\address{Emilia L.K. Bl{\aa}sten, Department of Computational Engineering, LUT University, Finland}
\email{emilia.blasten@iki.fi}

\author[]{Valter Pohjola}
\address{Valter Pohjola, Unit of Applied and Computational Mathematics, University of Oulu, Finland}
\email{valter.pohjola@gmail.com}

\begin{abstract}
We investigate the localization and vanishing of $L^2$ interior transmission eigenfunctions
at corners.
Past numerical computations suggest that these eigenfunctions localize at non-convex corners.
This phenomenon has, however, not been proven theoretically.
We show that localization does indeed occur for some  eigenfunctions at a non-convex corner.

We also investigate the vanishing of interior transmission eigenfunctions at a convex corner.
We prove that these eigenfunctions vanish at convex corners with reduced smoothness assumptions 
compared to earlier results. 
\end{abstract}

\maketitle

\tableofcontents

\section{Introduction} 

\noindent
The interior transmission problem  consists of finding $v,w \in L^2(D)$
and a wavenumber $k \in \C$, that solves the problem 
\begin{equation} \label{eq_ite}
\begin{aligned}
\begin{cases}
(\Delta + k^2) v = 0, \quad &\text{ in  }\quad D, \\
(\Delta + k^2n) w = 0, \quad &\text{ in }\quad D, \\
\p^j_\nu v = \p^j_\nu w, \quad &\text{  on }\quad\p D, \quad j=0,1, 
\end{cases}
\end{aligned}
\end{equation}
where $D \subset \R^n$ is a bounded Lipschitz domain, $n \in L^\infty(D)$ is the refractive index,
and $\nu$ is the outer unit normal vector to $\p D$.
The boundary condition is understood in the weak sense, and due to elliptic interior regularity can be replaced by the condition $v-w \in H^2_0(D)$.
We call $k$ an Interior Transmission Eigenvalue (ITEV) and the pair $(v,w)$ and the 
functions $v$ and $w$ Interior Transmission Eigenfunctions (ITEF).

The interior transmission problem arose originally in conjunction with 
reconstruction problems in scattering theory.
Certain numerical methods such as the linear sampling and factorization methods
need to avoid wave numbers $k$ that are  ITEVs  (see \cite{CKP89,Kr99,CCH02,CC06}), in order
to produce reliable reconstructions of a scatterer.

The interior transmission problem is also closely related to the phenomena 
of potentials that always produce a scattered wave (see e.g. \cite{BPS14,HSV16,LX17,BL21,SS21,BP22,CV23,CHLX24,HV25}).
A non-scattering incident wave and its corresponding total wave are ITEFs on the support of the potential.
Before \cite{BPS14} the converse was not known.
Special interest is given to real-valued transmission eigenvalue $k \in \R$, since these
are directly related to this type of inverse scattering problems.

\medskip
\noindent
Here we are interested in the behaviour of the ITEFs near corner points, and in particular
the localization phenomena for non-convex corners that has been observed in \cite{BLLW17,Liu20,CDHLW21,DLS21,DLWY21,DL24}.
Numerical computations have suggested that the ITEFs can become unbounded 
at the inward pointing corner point, in sharp contrast to e.g. Dirichlet eigenfunctions.
This type of localization at a non-convex corner is also in stark contrast to
the behaviour of ITEFs at convex corners. The latter vanish there. This has been proven in \cite{BL17a,BL17b,DLWY21} assuming extra regularity on the ITEFs and observed numerically in \cite{BLLW17,DLWY21}.
Our main task in this paper is to prove that some ITEFs do indeed become unbounded near an inwards pointing corner. This is the first proof of the singular and localization phenomenon for transmission eigenvalues at corners.

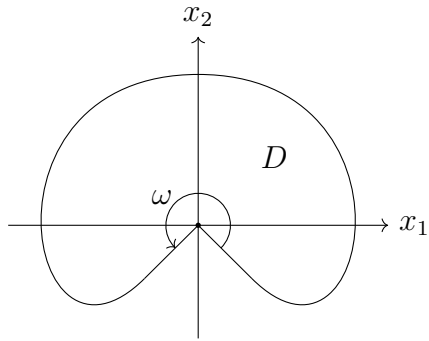
\begin{figure}[H]
  \begin{center}
    \begin{tikzpicture}
      \draw[->] (-2.5,0) -- (2.5,0) node[right] {$x_1$};
      \draw[->] (0,-1.5) -- (0,2.5) node[above] {$x_2$};
      \draw (0.7071,-0.7071) to (0,0);
      \draw (-0.7071,-0.7071) to (0,0);
      \draw (0.7071,-0.7071) to[out angle=-45, in angle=-135, curve through = {(0,2)}] (-0.7071,-0.7071); 
      \fill (0,0) circle (1pt);
      \draw (1,0.9) node {$D$};
      \draw[->] (0.3,-0.3) arc (-45:225:0.424) node[left, pos=0.6] {$\omega$};
    \end{tikzpicture}
    \caption{An example domain $D$ with a non-convex corner of opening angle $\omega > \pi$.
             At this type of corner the ITEF can become unbounded.}
    \label{fig_non_convex}
  \end{center}
\end{figure}

We focus on bounded domains $D \subset \R^2$  which have a single corner point $x_0$ and are otherwise
smooth. See figure \ref{fig_non_convex}. 
The opening angle at $x_0$ is denoted by $\omega$.
The localization phenomena occurs only for non-convex corners, i.e. when $\omega \in (\pi ,2\pi)$.
We further restrict $\omega$ so that $\omega \in (\omega_0, 2 \pi)$, where
$\omega_0 \approx 257.6^\circ$ is the solution to $\tan ( \omega) = \omega$. 
It seems that
our arguments should also work in the range $\omega \in (\pi, \omega_0)$
with slightly longer computations, we do not however pursue this case here.

Our main result on localization of ITEFs at non-convex corners is the following. 
This result is an immediate consequence of Proposition \ref{prop_main_loc}.

\begin{theorem} \label{thm_main_loc}
Let $\omega \in (\omega_0, 2\pi)$. Then there exists a corner domain $D$ satisfying \eqref{eq_def_D}, with a corner point at $x_0=0$,
and positive constants $n \neq 1$, and $k > 0$, such 
that the solutions $v$ and $w$ of \eqref{eq_ite}, are such that 
$$
v(r,\theta),w(r,\theta) \;\sim\;  C(\theta) r^{-\alpha}, \qquad   \text{ as } r \to 0, 
$$
for some $\alpha \in (0,1)$ and where $ C(\theta) \not \equiv 0$.
\end{theorem}

\noindent
Theorem \ref{thm_main_loc} gives the first examples of ITEFs where we can verify that 
localization occurs at a non-convex corners and thus going beyond the numerical results
in \cite{BLLW17,Liu20,CDHLW21,DLS21,DLWY21,DL24}.

\medskip
\noindent
We will also investigate the vanishing of ITEFs when the corner is convex, i.e. $\omega \in (0,\pi)$.
It appears that the localization of an ITEF is not possible at convex corners. In fact, they always seem to vanish instead.
For the convex corner case  we prove the following result, which is 
is a direct consequence of Proposition \ref{prop_simple_vanish}.

\begin{theorem} \label{thm_main_vanishing}
Let $\omega < \pi$, and let $D$ be domain of the form of \eqref{eq_def_D} with a corner point at $x_0$.
Assume that $n \in C^1(\overline D)$, $n>0$ and $n(0) \neq 1$, and let $v,w \in H^1(D)$ be solutions of \eqref{eq_ite}, 
Then $w(x_0) = 0$ in the following sense:
\[
  \lim_{\varepsilon \to 0} \frac{1}{|C_{\varepsilon}|} \int_{C_\varepsilon} w(x) dx = 0,
\]
where $C_{\varepsilon} = B(x_0,\varepsilon) \cap D$. The same applies also to $v$.
\end{theorem}

\noindent
Theorem \ref{thm_main_vanishing} improves on the results in \cite{BL17a,BL17b,Bla18,DLWY21}. In 
\cite{BL17a,BL17b,Bla18} it is assumed that the ITEFs $v$ and $w$ are in $H^2(D)$ (or H\"older) near the corner point. \cite{DLWY21} improves the approximation property of \cite{BL17a,BL17b}.
Theorem \ref{thm_main_vanishing} only assumes that
$v,w \in H^1(D)$. The question for $L^2(D)$-regularity remains open.

\medskip
\noindent
Let us briefly discuss the ideas behind the proof of Theorem \ref{thm_main_loc}.
Our main tool is the theory of elliptic PDE in corner domains. 
This theory goes back to the pioneering work by Kondratiev \cite{Ko67} with later
developments made by several authors, see e.g. \cite{Gr85,Gr92,NP94,MNP00}.
We will in particular use this type of results for the biharmonic equation 
which are obtained by Grisvard in \cite{Gr85} and \cite{Gr92}. 

We firstly augment the biharmonic theory in \cite{Gr85} and \cite{Gr92}
by deriving a condition that guarantees the appearance of a singular term
in the corner asymptotics derived in \cite{Gr85} and \cite{Gr92}. This analysis is carried out in 
section \ref{sec_a_condition_for_sing}.
Our main strategy is then to formulate the ITEV problem in the constant coefficients case.
This problem can be formulated as a fourth order problem that resembles a buckling type problem.
We then study an auxiliary eigenvalue problem related to the buckling type problem, see
subsections \ref{sec_buckling} and \ref{sec_auxiliary}.
This gives us  a family of eigenfunctions $ \{  \psi_j  \}$ of the 4th order problem that is complete in $H^2_0(D)$.
The eigenfunctions $\psi_j$ can be converted  into ITEFs in a natural way.
The completeness of the family $\{\psi_j\}$, and  the criterion for the occurrence of a singular 
term for a biharmonic source problem of Proposition \ref{prop_sing_cond_D} 
allows us then to find at least one $j$ for which $\Delta \psi_j$ becomes singular, see Lemma
\ref{lem_f_sing}.
This in turn translates to an ITEF that becomes singular.

We like point out that proving localization at a corner
involves some degree of difficulty since unboundedness of the ITEF always implies that the eigenfunction cannot come from a scattering solution due to elliptic regularity in $\R^2$. This means that none of the results such as those in \cite{BPS14,BP22,CV23,EH18,PSV17} can say anything about unbounded ITEFs.
In fact our results imply that $L^2$ interior transmission eigenfunctions forms a much more exotic class of functions than incident waves or total waves, something which was known previously only due to the existence of always scattering potentials that have ITEFs \cite{BPS14}.
We further discuss 
the connection to corner scattering in section \ref{sec_corner_scattering}.

\medskip
\noindent
The paper is structured as follows in section \ref{sec_pre} we review the results in corner asymptotics that we will
using in the later sections. In section \ref{sec_a_condition_for_sing} we derive a condition for 
the apperance of a certain singular term in the asymptotic expansion of solutions to the biharmonic equation.
In section \ref{sec_sing_ITEF} and its  subsections we construct a singular ITEF and prove Theorem \ref{thm_main_loc}.
Section \ref{sec_vanish} then shortly discuss the vanishing in convex corners and prove Theorem \ref{thm_main_vanishing}.
In the last section \ref{sec_corner_scattering} we discuss the connection to corner scattering and 
mention some open problems.

\section{Preliminaries} \label{sec_pre}

\noindent
In the following we will study bounded corner domains $D \subset \R^2$ satisfying
\begin{equation} \label{eq_def_D}
\begin{aligned}
D \cap B(0, \eps) = \mathcal C_\omega \cap B(0,\eps), \qquad  \p D \setminus \{ 0 \} \text{ is smooth },
\end{aligned}
\end{equation}
for some $\eps >0$  
where $\mathcal C_\omega$ is the cone with the opening angle $\omega \in (0,2\pi) \setminus \{\pi\}$ given by
\begin{equation} \label{eq_C_omega}
\begin{aligned}
\mathcal{C}_\omega := \{ (r,\theta) \in \R^2 \,:\, \theta \in (0, \omega), \, r>0 \}.
\end{aligned}
\end{equation}
See figure \ref{fig_non_convex} for an illustration
of a possible $D$.

\medskip
\noindent
We will need certain weighted Sobolev spaces
that are well adapted to study the asymptotics of solutions at corners, and 
which we will define next.
Let $\Omega \subset \R^2$ be a bounded Lipschitz domain.
We will consider the  weighted Sobolev spaces $V^\ell_\beta(\Omega) \subset L^2(\Omega)$. 
The norm specifying $V^\ell_\beta(\Omega)$, takes the weighted $L^2$-norms of derivatives up to order $\ell$, i.e.
\begin{equation} \label{eq_V_l_b}
\begin{aligned}
\| u \|^2_{V^\ell_\beta( \Omega)} := 
\sum_{|\alpha| \leq \ell} \big\| r^{\beta-\ell+|\alpha|} \p^\alpha_x u  \big\|^2_{L^2( \Omega)}.
\end{aligned}
\end{equation}
See also p. 27 in \cite{NP94}.
It will be convenient to also characterize the $V^\ell_\beta(\Omega)$-norm in polar coordinates.
This reveals that the norm considers $r^{-1}$ and $\p_r$ as somewhat equivalent.

\begin{lemma} \label{lem_V_polar}
We have that $u \in V^\ell_\beta(\Omega)$ if and only if 
$$
r^{\beta - \ell + j}|\p^j_r\p^k_\theta u(r,\theta) | \in L^2(\Omega), \quad \text{ for } \quad j + k  \leq \ell.
$$
\end{lemma}

\begin{proof}
The claim follows by a direct computation using polar coordinates.
\end{proof}

\noindent
The following can be found in \cite{Gr85} see Theorem 7.2.1.1. 

\begin{lemma} \label{lem_H20_in_V20}
If $w \in H^2_0(\mathcal C)$, then $w \in V^2_0(\mathcal C)$.
\end{lemma}	

\noindent
Our main tool in studying the localization of ITEFs in section \ref{sec_sing_ITEF} will be the following result 
describing the corner asymptotics of a solution to the biharmonic source problem in an cone.
The following proposition is a direct consequence of Theorem 3.2.10 in \cite{Gr92} and
Theorem 7.2.1.12 \cite{Gr85}. 
(See also Theorem 3.4.1 in \cite{Gr92} and the discussion after Theorem 7.2.1.12 in \cite{Gr85}.)

\begin{proposition} \label{prop_Biharmonic_asymptotics}
Suppose $u \in V^2_0(\mathcal C_\omega)$ is of bounded support and solves
\begin{equation} \label{eq_biharmonic_with_source}
\begin{aligned}
\left\{ 
  \begin{aligned}
    \Delta^2 u &= f, \quad && \text{in } \mathcal C_\omega,\\
    \p^j_\nu u |_{\p \mathcal C_\omega} &= 0, \quad && j=0,1,
  \end{aligned}
\right.
\end{aligned}
\end{equation}
where $f \in V^0_0(\mathcal C_\omega)$. 
Assume that $\omega \in (0,2\pi)$ is such that $\tan \omega \neq \omega$, and that the equation
\begin{equation} \label{eq_char_eq}
\begin{aligned}
\sin^2(z \omega) = z^2 \sin(\omega),
\end{aligned}
\end{equation}
has no solution $z \in \C$ with $\Re z = 2$. Then $u$ is of the form
$$
u = u_R + u_S, \qquad u_R \in V_0^{4}(\mathcal C_\omega),
$$
where
\begin{equation} \label{eq_sing_expansion}
\begin{aligned}
u_S(r,\theta) = 
\sum_{m=1}^{M'} c_m \varphi_m(\theta) r^{1 + z'_m} 
+ 
\sum_{m=1}^{M''}  [ c'_m \phi_m(\theta) +  c''_m\ln r \varphi_m(\theta)] r^{1 + z''_m},
\end{aligned}
\end{equation}
and $z'_m$ and $z''_m$ are respectively simple and double roots  of \eqref{eq_char_eq}  
with $ \Re z'_m, \Re z''_m \in  (0,2)$, and $\phi_m,\varphi_m \in C^\infty(0,\omega)\cap H^2_0(0,\omega)$
are normalized as in \eqref{eq_normalization}, and $c_m,c'_m,c''_m \in \C.$
\end{proposition}

\begin{remark} \label{rem_ODE}
The functions $\varphi_m$ and $\phi_m$ in Proposition \ref{prop_Biharmonic_asymptotics} are normalized by
\begin{equation} \label{eq_normalization}
\begin{aligned}
  \| \varphi_m \|_{L^2(0,\,\omega)},\;\| \phi_m \|_{L^2(0,\,\omega)} = 1,
\end{aligned}
\end{equation}
and they can be described explicitly as solutions to certain ODEs
(see Theorem 3.4.1 and Theorem 3.2.10 in \cite{Gr92}). The functions $\varphi_m \in H^2_0(0,\omega)$ 
solve the ordinary differential equations
$$
\varphi'''' + 2(1+z^2) \varphi''  + (z^4-2z^2+1) \varphi = 0, 
$$
where $z = z'_m$.
The functions $\phi_m \in H^2_0(0,\omega)$, solve in turn the ODE
$$
 \phi''''_m + 2(1+z^2) \phi''_m  + (z^4-2z^2+1) \phi_m = 
-4(z^3 - z)\varphi_m +  4 z_m \varphi''_m,
$$
where $z = z''_m$.
\end{remark}

\noindent
\begin{remark} \label{rem_constant_non_zero}
Note that Proposition \ref{prop_Biharmonic_asymptotics} gives a  singular expansion near the vertex of the cone
to a solution $u$ of the biharmonic source problem, but it does not give a condition when the singular terms $c_mr^{1+z'_m}\varphi_m$, 
$c'_m r^{1+z''_m}\phi_m$
and $c''_m r^{1+z''_m} \ln(r) \varphi_m$ are non-zero. The constants $c_m,c'_m,c''_m$ could be zero.
We will need such a condition and  deal with this in section \ref{sec_a_condition_for_sing}.
\end{remark}

\medskip
\noindent
We shall also --- in section \ref{sec_vanish} --- study the vanishing of ITEFs at convex corners. For this 
we will use the corner singularity theory for the Laplace operator. 
Will need the following proposition, which is a variant of Theorem 3.1 p. 32 in \cite{NP94}.

\begin{proposition} \label{prop_V_Dirichlet}
Let $f \in V^0_\beta(D)$, and $|\beta - 1| < \frac{\pi}{\omega}$. Then there exists a unique solution $u \in V^2_\beta(D)$
to 
\begin{align*}
\left\{
  \begin{aligned}
    \Delta u &= f , \quad && \text{in } D\\
    u|_{\p D} &= 0, \quad &&
  \end{aligned}
  \right.
\end{align*}
and the estimate 
$$
\| u \|_{V^2_\beta (D )} \leq C \| f \|_{V^0_\beta (D )}
$$
holds.
\end{proposition}

\section{A condition for the appearance of a singular term} \label{sec_a_condition_for_sing}

\noindent 
In this section we further investigate the corner singularity
result in \cite{Gr92} for the biharmonic equation, which we stated in Proposition \ref{prop_Biharmonic_asymptotics}.
We  will need to expand on this result by finding a suitable 
criterion for when a specific singular term in $u_S$ is non-zero, and also find a way to apply it to 
the bounded domain $D$. For this we use similar ideas\footnote{
\cite{NP94} deals with the Laplacian case. The proof of the condition in the Laplacian case
uses the orthogonality of the $\varphi$ and $\phi$
which in that case are just sine functions.}
as in \cite{NP94} Chapter 2.

We will assume that $u \in V^2_0(\mathcal C_\omega)$ is of bounded support and solves
\begin{equation} \label{eq_biharmonic_with_source_SEC_condition}
\begin{aligned}
\begin{cases}
\Delta^2 u = f, \qquad \text{ in } \mathcal C_\omega,\\
\p^j_\nu u |_{\p \mathcal C_\omega} = 0, \quad j=0,1,
\end{cases}
\end{aligned}
\end{equation}
where $f \in V^0_0(\mathcal C_\omega)$. By Proposition \ref{prop_Biharmonic_asymptotics}
we can write $u=u_S + u_R$ as in \eqref{eq_sing_expansion}. We will now investigate the
singular part $u_S$ in more detail.

\medskip
\noindent
We shall focus on the term behaving worst, in terms of decay in $r$, in the expansion \eqref{eq_sing_expansion}.
To single out this term we will investigate the location in the complex plane, of the roots $z_k = z'_k,z''_k$, that appear in
the expansion \eqref{eq_sing_expansion}.
The following Lemma lets us characterize $z_1$, which gives us the worst behaving term in \eqref{eq_sing_expansion} in terms
of decay in $r$ at zero.
For a proof see \cite{Gr92} Lemma 3.3.2 p. 100.

\begin{lemma} \label{lem_z1_real}
The characteristic equation 
\begin{equation} \label{eq_char_eq_2}
\begin{aligned}
\sin^2(z \omega) = z^2 \sin(\omega),
\end{aligned}
\end{equation}
in Proposition \ref{prop_Biharmonic_asymptotics},  
has only the following roots $z_j$ in the strip $ (0,1) \times \R \subset \C$
  \begin{enumerate}[label=(\alph*)]
  \item \label{eq_roots_double} a single real double root $z_1$ when $\omega \in (\pi, \omega_0)$,
  \item \label{eq_roots_single} two real simple roots $z_1 < z_2$ when $\omega \in (\omega_0, 2\pi)$.
  \end{enumerate}
Here $\omega_0 \in (\pi,2\pi)$ solves $\tan(\omega_0) = \omega_0$. We have $\omega_0 \approx 257.6^\circ$.
\end{lemma}

\noindent 
We will now consider the case of non-convex angles with  $\omega \in (\omega_0, 2\pi)$,
in which case we can let the root $z'_1$ in the expansion \eqref{eq_sing_expansion} be such that $z'_1=z_1$,
where $z_1$ is from case~\ref{eq_roots_single} of Lemma~\ref{lem_z1_real}.
Lemma \ref{lem_biharmonic_c1_not_zero} below gives a condition for when the singular term in \eqref{eq_sing_expansion} 
corresponding to $z'_1$ is non-zero.

To state Lemma \ref{lem_biharmonic_c1_not_zero} we need to introduce the function $\eta_1$, which we define as
\begin{equation} \label{eq_eta_1}
\begin{aligned}
\eta_1 := r^{1-z_1} \varphi_1(\theta), \qquad (r,\theta) \in \mathcal C_\omega,
\end{aligned}
\end{equation}
where $\varphi_1$ corresponds to the term in the expansion \eqref{eq_sing_expansion}. Recall that by the remark \eqref{rem_ODE} it solves
(see also \cite{Gr92} Theorem 3.2.10) the ODE
\begin{equation} \label{eq_ODE_1}
\begin{aligned}
\varphi'''' + 2(1+z^2)  \varphi''  + (z^4-2z^2+1) \varphi = 0, \qquad \varphi \in H^2_0(0,\omega),
\end{aligned}
\end{equation}
with $z =  z_1$. 
An important property of $\eta_1$
is that it is biharmonic in the interior of the cone $\mathcal C_\omega$.
Note that we can also pick $z=-z_1$ in \eqref{eq_ODE_1}. This and
a direct computation in polar coordinates and the fact that $\varphi_1$ solves \eqref{eq_ODE_1} show that
\begin{equation} \label{eq_Delta2_eta1_0}
\begin{aligned}
\Delta^2 \eta_1 (x) = 0, \qquad \text{ for } x 	\in \mathcal C_ \omega.
\end{aligned}
\end{equation}
We are now ready to derive a condition for a singular term corresponding to a simple root $z'_1=z_1$ of the characteristic equation
\eqref{eq_char_eq_2}.

\begin{lemma} \label{lem_biharmonic_c1_not_zero}
Let $f \in L^2(\mathcal C_\omega)$ with bounded support. Assume
that $u$ is a solution of \eqref{eq_biharmonic_with_source_SEC_condition} with bounded support, and
assume that $\omega \in (\omega_0, 2\pi)$, and that $z_1 \in \R$ is the smaller of the real roots of equation \eqref{eq_char_eq_2}
in Lemma \ref{lem_z1_real}. Then the singular term $u_S$ in \eqref{eq_sing_expansion} has
\begin{equation} \label{eq_sing_cond}
\begin{aligned}
c_1 \neq 0, \qquad \text{ if } \quad (f,\eta_1)_{L^2(\mathcal C_\omega)} \neq 0,
\end{aligned}
\end{equation}
where $\eta_1$ is given by \eqref{eq_eta_1}.
\end{lemma}

\begin{proof}
Assume that $(f,\eta_1)_{L^2(\mathcal C_\omega)} \neq 0$.
We need to show that $c_1 \not = 0$. We will do this by showing  that $c_1$ can be expressed in terms of  
the inner product $(f,\eta_1)_{L^2}$.
Let $B_\eps := B(x_0, \eps)$, $\eps > 0$. Since $f$ has bounded support and $\eta_1$ is locally in $L^2$, we have
\begin{align} \label{eq_f_eta1_lim}
(f,\eta_1)_{L^2( \mathcal C_\omega)} = 
\lim_{\eps \to 0}(f,\eta_1)_{L^2( \mathcal C_\omega \setminus B_\eps)}.
\end{align}
Set $\mathcal C_\eps := \mathcal C_\omega \setminus B_\eps$ and $\Gamma_\eps := \mathcal C_\omega \cap \{ r=\eps  \}$.
Keeping in mind that $\eta_1$ solves the biharmonic equation in $\mathcal C_\eps$, because of \eqref{eq_Delta2_eta1_0}, 
we have by integrating by parts, that 
\begin{equation} \label{eq_as_bndry_ints}
\begin{aligned}
(f,\eta_1)_{L^2(\mathcal C_\eps )}
&=
(\Delta^2 u ,\eta_1)_{L^2(\mathcal C_\eps )} \\
&= 
(\p_r \Delta u ,\eta_1)_{L^2(\Gamma_\eps )}
- (\Delta u ,\p_r\eta_1)_{L^2(\Gamma_\eps )} \\
&+ (\p_r u ,\Delta \eta_1)_{L^2(\Gamma_\eps )}
- (u , \p_r\Delta  \eta_1)_{L^2(\Gamma_\eps )}.
\end{aligned}
\end{equation}
Using the expansion \eqref{eq_sing_expansion}, we can write
\begin{equation}  \label{eq_sing_expansion_2}
\begin{aligned}
u(r,\theta) =  
\sum_{m=1}^{M'} c_m \varphi_m(\theta) r^{1 + z'_m} 
+ 
\sum_{m=1}^{M''} [c'_m \phi_m(\theta) + c''_m \ln r \varphi_m(\theta)] r^{1 + z''_m} + u_R,
\end{aligned}
\end{equation}
where $u_R \in V_0^{4}(\mathcal C_\omega)$, and $z'_m$ and $z''_m$ are simple and double roots of equation
\eqref{eq_char_eq_2}.
By Proposition~\ref{prop_Biharmonic_asymptotics} we have $\Re z_m', \Re z_m''
\in (0,2)$ for all $m$. Let's reorder the roots in increasing order of real
parts, so that $z_1'$ has the smallest real part of the single roots.
Lemma~\ref{lem_z1_real} guarantees that $z_1' = z_1$ from case (b) therein, and
that $\Re z_1' \in (0,1)$. The terms in \eqref{eq_sing_expansion_2} involving
$z_1'$ are the most singular.

Applying Lemmas \ref{lem_bndry_ints} and \ref{lem_sing_bndry_ints}, we see that terms other than $c_1 \varphi_1(\theta) r^{1 + z'_1}$ do not contribute to the limit \eqref{eq_f_eta1_lim},
and that therefore
\begin{equation*} 
\begin{aligned}
\lim_{\eps \to 0}(f,\eta_1)_{L^2( \mathcal C_\omega \setminus B_\eps)}
= \lim_{\eps \to 0} \Big(
&(\p_r \Delta [c_1r^{1+z_1} \varphi_1] ,\eta_1)_{L^2(\Gamma_\eps )}
- (\Delta [c_1r^{1+z_1} \varphi_1]  ,\p_r\eta_1)_{L^2(\Gamma_\eps )} \\
&+ (\p_r  [c_1r^{1+z_1} \varphi_1] ,\Delta \eta_1)_{L^2(\Gamma_\eps )}
- ( c_1r^{1+z_1} \varphi_1 , \p_r \Delta  \eta_1)_{L^2(\Gamma_\eps )} \Big).
\end{aligned}
\end{equation*}
Let us begin by analyzing the contribution of the term containing $\p_r\Delta [\varphi_1(\theta) r^{1 + z'_1}]\eta_1$.
For this we have that
\begin{align*}
\Big(\p_r \Delta [r^{1+z_1}\varphi_1] ,r^{1-z_1}\varphi_1  \Big)_{L^2(\Gamma_\eps )}
=
(z_1-1)(1+z_1)^2 \| \varphi_1 \|^2_{L^ 2(0,\omega)} 
+(z_1-1) (\varphi''_1,\varphi)_{L^ 2(0,\omega)}.
\end{align*}
Thus we have that 
\begin{align*}
(\p_r \Delta u ,\eta_1)_{L^2(\Gamma_\eps )}
&=
(z_1-1)(z_1+1)^2 \| \varphi_1 \|^2_{L^ 2(0,\omega)} 
+(z_1-1) (\varphi''_1,\varphi)_{L^ 2(0,\omega)}  
+ o(1),
\end{align*}
as $\eps \to 0$.
Carrying out a similar analysis for the three remaining boundary terms in \eqref{eq_as_bndry_ints},
yields 
\begin{align*}
- (\Delta u ,\p_r\eta_1)_{L^2(\Gamma_\eps )} 
&=
(z_1-1)(z_1+1)^2 \| \varphi_1 \|^2_{L^ 2(0,\omega)} 
+(z_1-1) (\varphi''_1,\varphi)_{L^ 2(0,\omega)}  
+ o(1), \\
(\p_r u ,\Delta \eta_1)_{L^2(\Gamma_\eps )}
&=
(z_1+1)(z_1-1)^2 \| \varphi_1 \|^2_{L^ 2(0,\omega)} 
+(z_1+1) (\varphi''_1,\varphi)_{L^ 2(0,\omega)}  
+ o(1),\\
- (u ,\p_r\Delta  \eta_1)_{L^2(\Gamma_\eps )}
&=
(z_1+1)(z_1-1)^2 \| \varphi_1 \|^2_{L^ 2(0,\omega)} 
+(z_1+1) (\varphi''_1,\varphi)_{L^ 2(0,\omega)}  
+ o(1),
\end{align*}
as $\eps \to 0$.
Putting these together yields that 
\begin{align*}
(f,\eta_1)_{L^2(\mathcal C_\omega )}
&=
c_1 \Big(
4z_1(z_1^2-1) \| \varphi_1 \|^2_{L^ 2(0,\omega)}  - 4z_1\| \varphi'_1\|^ 2_{L^ 2(0,\omega)} 
\Big).
\end{align*}
Note that the term multiplying $c_1$ on the right is less than zero, since $\Re z_1 \in (0,1)$.
It follows that $c_1 \neq 0$, if $(f,\eta_1)_{L^2(\mathcal C_\omega )} \neq 0$.
\end{proof}

\begin{remark} Note that we could attmept, by 
the procedure used in Lemma \ref{lem_biharmonic_c1_not_zero},
derive similar conditions as \eqref{eq_sing_cond} for the other singular terms (i.e. when $z'_m \neq z_1$) in the 
expansion  \eqref{eq_sing_expansion}. For the root $z_2$, with the second smallest
real part, we could argue as in the proof of Lemma \ref{lem_biharmonic_c1_not_zero}, 
where we use $u - c_1 r^{1+z_1} \varphi_1$, in place of $u$, and replace $\eta_1$ with $\eta_2 = r^{1-z_2} \varphi_2$.
One could then proceed as in the proof of Lemma \ref{lem_biharmonic_c1_not_zero}, since
also $\Delta^2 [u - c_1 r^{1+z_1} \varphi_1] = f$.
Iterating this argument should give the conditions for all the singular terms. 
We do not however pursue this here. 
\end{remark}

\noindent
To complete the proof of Lemma \ref{lem_biharmonic_c1_not_zero} we need to derive 
the two lemmas used to evaluate the boundary integrals in its  proof. 
Before this it will be convenient to derive the following scaled version of the trace theorem.

\begin{lemma} \label{lem_scaled_trace}
  Let $u \in H^1(D)$ where $D$ is defined in \eqref{eq_def_D}, and let $\Gamma_\eps := D \cap \{  r=\eps  \}$ be a non-empty circular arc.
Then for small $\eps >0 $, we have that
\begin{align*}
\| w \|_{L^2(\Gamma_\eps)} \leq C \big( \eps^{-1/2}\| u \|_{L^2(D)} + \eps^{1/2}\| \nabla u \|_{L^2(D)} \big),
\end{align*}
where $C$ does not depend on $\eps$.
\end{lemma}

\begin{proof}
We can assume that $\eps_0 > \eps >0$, where $\eps_0>0$ is small enough so that $D$ coincides
with the cone, and we have that the circular arc $\Gamma_0 := D \cap \{ r = \eps_0 \}$, and
$D_0 := D\cap B(0,\eps_0) = \mathcal C_\omega \cap B(0,\eps_0)$.
Note firstly that $\Gamma_\eps$ is one dimensional. We integrate by substitution, as follows
\begin{align*}
\| w \|^2_{L^2(\Gamma_\eps)}  
=  \int_{\Gamma_\eps} | w(x) |^2 \, dS  
=  \int_{\Gamma_0} | w(\eps x) |^2 \eps \, dS. 
\end{align*}
By the trace theorem for $\Gamma_0 \subset \overline D_0$, we have that
\begin{align*}
\eps \| w(\eps x) \|^2_{L^2(\Gamma_0)}  
\leq C\eps  \| w(\eps x) \|^2_{H^1(D_0)}, 
\end{align*}
where $C$ does not depend on $\eps$. 
Note that $D_0$ is 2 dimensional and that by integration by substitution we have that
\begin{align*}
\eps \| w(\eps x) \|^2_{L^2(D_0)} 
=  \eps^{-1} \int_{D_0} | w(\eps x) |^2 \eps^2  \, dx  
=  \eps^{-1} \int_{D_\eps} | w( x) |^2\, dx,  
\end{align*}
where $D_\eps := D \cap B(0,\eps)$. Similarly
\begin{align*}
\eps \| \nabla w(\eps x) \|^2_{L^2(D_0)} 
=  \eps \int_{D_0} | (\nabla w )(\eps x) |^2 \eps^2  \, dx  
=  \eps \int_{D_\eps} | \nabla w( x) |^2\, dx.  
\end{align*}
Putting these estimates together yields 
\begin{align*}
\| w \|^2_{L^2(\Gamma_\eps)}  
\leq 
C\Big( \eps^{-1} \int_{D_\eps} | w( x) |^2\, dx +  \eps \int_{D_\eps} | \nabla w( x) |^2\, dx\Big), 
\end{align*}
from which the claim follows.
\end{proof}

\noindent
Now we use the fact that functions in $V^4_0(D)$ have to behave well as we approach the corner point,
and prove the first lemma that Lemma \ref{lem_biharmonic_c1_not_zero} depends on.

\begin{lemma} \label{lem_bndry_ints}
Suppose $w \in V^4_0( \mathcal C_\omega)$, and define $\Gamma_\eps := \mathcal C_\omega \cap \{ r = \eps \}$. Then
we have that
\begin{align*}
\begin{rcases*}
(\p_r \Delta w ,\eta_1)_{L^2(\Gamma_\eps )} \\
(\Delta w ,\p_r\eta_1)_{L^2(\Gamma_\eps )}  \\
(\p_r w ,\Delta \eta_1)_{L^2(\Gamma_\eps )} \\
(w , \p_r\Delta  \eta_1)_{L^2(\Gamma_\eps )} 
\end{rcases*}
=  o(1), \qquad  \text{ as } \eps \to 0,
\end{align*}
where $\eta_1 := r^{1-z_1}\varphi_1$.
\end{lemma}

\begin{proof}
Consider first the term 
\begin{align*}
(w ,\p_r\Delta  \eta_1)_{L^2(\Gamma_\eps )} 
&= 
\int_0^\omega  w  \p_r \Delta  [ r^{1-z_1} \varphi_1] r  d\theta  \Big|_{r= \eps} \\ 
&= 
\int_0^\omega  w  r^{-1-z_1} [(1-z_1^2)(z_1 -1)\varphi_1 - ( 1+z_1)\varphi''_1 ] d\theta \Big|_{r= \eps} . 
\end{align*}
Let $\tilde \varphi := (1-z_1^2)(z_1 -1)\varphi_1 - ( 1+z_1)\varphi''_1$.
We can then estimate the  above using Lemma \ref{lem_scaled_trace}, as follows
\begin{align*}
(w , \p_r\Delta  \eta_1)_{L^2(\Gamma_\eps )} 
&= 
\int_0^\omega  w  r^{-2-z_1} \tilde \varphi(\theta) \,r \, d\theta \Big|_{r= \eps} \\
&= 
\int_{\Gamma_\eps}  w  r^{-2-z_1} \tilde \varphi  dS \\
&\leq C 
\| r^{-2-z_1} w \|_{L^2(\Gamma_\eps)}  \| \tilde \varphi \|_{L^2(\Gamma_\eps)} \\ 
&\leq C 
\eps^{-1/2}\| r^{-2-z_1} w \|_{H^1(D)}  \| \tilde \varphi \|_{L^2(\Gamma_\eps)} \\
&\leq C 
 \eps^{-1/2}\| w \|_{V^4_0(D)}  \| \tilde \varphi \|_{L^2(\Gamma_\eps)}.
\end{align*}
Notice that $\tilde \varphi$ is a function of $\theta$ and $C$ does not depend on $\eps$.
The norm of $\| \tilde \varphi \|_{L^2(\Gamma_\eps)} \sim \eps$,
as $\eps \to 0$. We thus see that
\begin{align*}
(w ,\p_r \Delta  \eta_1)_{L^2(\Gamma_\eps )}  = o (1), \qquad \text{ as } \eps \to 0.
\end{align*}
The other cases can verified by similar computations.
\end{proof}

\begin{lemma} \label{lem_sing_bndry_ints}
Let $\Gamma_\eps$, $z_1$ and $\eta_1$ be as in Lemma \ref{lem_bndry_ints}.
Assume that $z \in \R \times (0,2)$ and $\Re z > \Re z_1$,
and suppose that 
$$
w = r^{1+z} h(\theta) \qquad \text{ or } \qquad  w = \ln(r) r^{1+z} h(\theta), 
$$
where $h \in H^2_0(0,\omega)$. Then
\begin{align}  \label{eq_r_bndry_terms}
\begin{rcases}
(\p_r \Delta w ,\eta_1)_{L^2(\Gamma_\eps )} \\
(\Delta w ,\p_r\eta_1)_{L^2(\Gamma_\eps )}  \\
(\p_r w ,\Delta \eta_1)_{L^2(\Gamma_\eps )} \\
(w ,\p_r\Delta  \eta_1)_{L^2(\Gamma_\eps )} 
\end{rcases}
=  o(1), \qquad  \text{ as } \eps \to 0.
\end{align}
\end{lemma}

\begin{proof}
Let us first consider the non-logarithmic case $w = r^{1+z}h(\theta)$. 
We begin with computing the first boundary term 
\begin{equation} \label{eq_1st_bndry_term}
\begin{aligned}
\big(\p_r \Delta w ,\eta_1\big)_{L^2(\Gamma_\eps )}
&= 
\big(\p_r \Delta [r^{1+z} h(\theta) ] ,r^{1-z_1}\varphi_1(\theta)\big)_{L^2(\Gamma_\eps )} \\
&= 
\int_0^\omega \p_r \Delta [r^{1+z} h(\theta) ] r^{1-z_1}\varphi_1(\theta) rd\theta.
\end{aligned}
\end{equation}
Now $z - z_1 = a+ib$, with $a>0$, so that
\begin{align*}
\int_0^\omega  \p_r \Delta [r^{1+z} h(\theta) ] r^{1-z_1}\varphi_1(\theta) rd\theta
&=
\int_0^\omega  r^{z-1} \tilde h(\theta) r^{1-z_1}\varphi_1(\theta) d\theta  \\
&=
\int_0^\omega r^{a+ib} \tilde h(\theta) \varphi_1(\theta) d\theta 
= o(1),
\end{align*}
as $r=\eps \to 0$, and where $\tilde h(\theta) = (z-1) \big( (1+z)^2 h(\theta) + h''(\theta) \big)$.

We continue by estimating the next term in \eqref{eq_r_bndry_terms}.
Again using the fact $z - z_1 = a+ib$ with $a > 0$, we have that
\begin{equation*}  
\begin{aligned}
\big(\Delta w ,\p_r \eta_1\big)_{L^2(\Gamma_\eps )}
&= 
\big(\Delta [r^{1+z} h(\theta) ] , (1-z_1) r^{-z_1}\varphi_1(\theta)\big)_{L^2(\Gamma_\eps )} \\
&= 
  (1-z_1) \int_0^\omega \Delta [r^{1+z} h(\theta) ] r^{-z_1}\varphi_1(\theta) rd\theta \\
&=
\int_0^\omega r^{z-1} \tilde h(\theta) r^{1-z_1}\varphi_1(\theta) d\theta  \\
&=
\int_0^\omega r^{a+ib} \tilde h(\theta) \varphi_1(\theta) d\theta 
= o(1),
\end{aligned}
\end{equation*}
as $r=\eps \to 0$, and where $\tilde h(\theta) = (1-z_1) \big( (1+z)^2 h(\theta) + h''(\theta) \big)$.

Continuing with the third and fourth boundary term \eqref{eq_r_bndry_terms}, we see similarly  that
\begin{equation*} 
\begin{aligned}
\big( \p_r w ,\Delta\eta_1\big)_{L^2(\Gamma_\eps )}
&= 
\int_0^\omega \p_r r^{1+z} h(\theta)  \Delta[r^{1-z_1}\varphi_1(\theta)] rd\theta \\
&=
\int_0^\omega r^{a+ib} h(\theta) \tilde \varphi_1(\theta) d\theta 
= o(1),
\end{aligned}
\end{equation*}
where $\tilde \varphi_1(\theta) = (1+z) \big( (1-z_1)^2 \varphi_1(\theta) + \varphi_1''(\theta) \big)$, and that
\begin{equation*}
\begin{aligned}
\big(  w ,\p_r\Delta\eta_1\big)_{L^2(\Gamma_\eps )}
&= 
\int_0^\omega  r^{1+z} h(\theta)  \p_r\Delta[r^{1-z_1}\varphi_1(\theta)] rd\theta  \\
&=
\int_0^\omega r^{a+ib} h(\theta) \tilde \varphi_1(\theta) d\theta 
= o(1),
\end{aligned}
\end{equation*}
as $r=\eps \to 0$, where $\tilde \varphi_1(\theta) = (-1-z_1) \big( (1-z_1)^2 \varphi_1(\theta) + \varphi_1''(\theta) \big)$.

\medskip
\noindent
Next, we move onto the case  $w = \ln(r)r^{1+z}h(\theta)$. These computations are to a large extent
identical. The first boundary tern in \eqref{eq_r_bndry_terms} looks like
\begin{equation*} 
\begin{aligned}
\big(\p_r \Delta w ,\eta_1\big)_{L^2(\Gamma_\eps )}
&= 
\big(\p_r \Delta [ \ln(r)r^{1+z} h(\theta) ] ,r^{1-z_1}\varphi_1(\theta)\big)_{L^2(\Gamma_\eps )} \\
&= 
\int_0^\omega \p_r \Delta [ \ln(r)r^{1+z} h(\theta) ] r^{1-z_1}\varphi_1(\theta) rd\theta.
\end{aligned}
\end{equation*}
Now $z - z_1 = a+ib$, with $a>0$, so that
\begin{align*}
\int_0^\omega  \p_r \Delta [\ln(r)r^{1+z} h(\theta) ] r^{1-z_1}\varphi_1(\theta) rd\theta
&=
\int_0^\omega  r^{z-1} (\tilde h_1(\theta) + \ln(r) \tilde h_2(\theta) ) r^{1-z_1}\varphi_1(\theta) d\theta  \\
&=
\int_0^\omega r^{a+ib} (\tilde h_1(\theta) + \ln(r) \tilde h_2(\theta) )  \varphi_1(\theta) d\theta 
= o(1),
\end{align*}
as $r=\eps \to 0$, and where we have $\tilde h_1(\theta) = (3z^2+2z-1)h(\theta) + h''(\theta)$ and $\tilde h_2(\theta) = (z-1) \big( (z+1)^2 h(\theta) + h''(\theta) \big)$.

The other boundary terms with $w = \ln(r)r^{1+z}h(\theta)$ can be evaluated similarly. 
\end{proof}

\subsection{The condition for a singular term on the bounded domain $D$} \label{sec_a_condition_for_sing_D}
\noindent
In this subsection
we derive a condition for the appearance of a non-zero singular term for solutions of the biharmonic source problem 
on  the domain $D$. This result is similar to the case of the cone $\mathcal C_\omega$ given by 
Lemma \ref{lem_biharmonic_c1_not_zero}.
The derivation is to large extent similar as the arguments in the previous section.
We need, however, to correct the $\eta_1$ that is defined by \eqref{eq_eta_1} on the cone  $\mathcal C_\omega$,
with a correction term $h$, so that when we add $h$ to $\eta_1$ we obtain a solution
to the biharmonic Dirichlet problem on $D$ with the zero boundary conditions
on the entire boundary $\p D$.

We begin by deriving the analogue of Proposition \ref{prop_Biharmonic_asymptotics}
for the bounded domain $D$. In the case of a bounded domain $D$, we need to define an analogue
to the $\eta_1$ that we defined in \eqref{eq_eta_1} for the cone.
Let $\chi \in C^\infty_0(\R^2)$ be a cut-off function satisfying
\begin{equation} \label{eq_chi_def}
\begin{aligned}
\chi(r) = 
\begin{cases}
1, \quad r \in (0,r_0/2), \\
0, \quad  r \in (2r_0,\infty),
\end{cases} \qquad r = \lvert x-x_0 \rvert
\end{aligned}
\end{equation}
where $B(x_0,2r_0) \cap C_\omega \subset D$.

\begin{proposition} \label{prop_singular_expansion_D}
Let $f \in V^0_0(D)$ and  $u \in V^2_0(D)$ solve the equation
\begin{equation} \label{eq_biharmonic_with_source_D_3}
\begin{aligned}
\begin{cases}
\Delta^2 u = f, \qquad \text{ in } D,\\
\p^j_\nu u |_{\p D} = 0, \quad j=0,1.
\end{cases}
\end{aligned}
\end{equation}
Then we have the expansion $u = u_R + u_S$, where
$u_R  \in V_0^{4}(D)$ and 
\begin{equation} \label{eq_sing_expansion_2new}
\begin{aligned}
u_S(r,\theta) = 
\chi \sum_{m=1}^{M'}  c_m \varphi_m(\theta) r^{1 + z'_m} 
+ 
\chi \sum_{m=1}^{M''}  [  c'_m \phi_m(\theta) +   c''_m\ln r \varphi_m(\theta)] r^{1 + z''_m},
\end{aligned}
\end{equation}
where $\chi$ is given by \eqref{eq_chi_def}, and
$z'_m$ and $z''_m$ and $\phi_m,\varphi_m \in C^\infty(0,\omega)\cap H^2_0(0,\omega)$
are as in \eqref{eq_sing_expansion}
and $ c_m, c'_m, c''_m \in \C.$
\end{proposition}
 
\begin{proof}
  Let $E_0$ be the operator that extends by zero functions defined on $D$ to $C_{\omega}$. The function $E_0 \chi u$ solves the following  biharmonic 
Dirichlet problem in the cone $\mathcal C_\omega$
\begin{align*}
\begin{cases}
  \Delta^2 (E_0\chi u) &= E_0(\chi f + [\Delta^2, \chi] u), \qquad \text{ in } \mathcal C_\omega. \\
  \p^j_\nu (E_0\chi u)|_{\p \mathcal C_\omega} &= 0.
\end{cases}
\end{align*}
Note that the source term is an element of $V_0^0(C_{\omega})$ and $E_0 \chi u \in V_0^2(C_{\Omega})$ since $f \in V_0^0(D)$ and $u \in V_0^2(D)$.
From Proposition \ref{prop_Biharmonic_asymptotics}
it follows that
\begin{equation*} 
\begin{aligned}
E_0 \chi u = \tilde u_R + 
\sum_{m=1}^{M'} c_m \varphi_m(\theta) r^{1 + z'_m} 
+ 
\sum_{m=1}^{M''}  [ c'_m \phi_m(\theta) +  c''_m\ln r \varphi_m(\theta)] r^{1 + z''_m},
\end{aligned}
\end{equation*}
where $\tilde u_R \in V^4_0(\mathcal C_\omega)$. Multiplying with a further $\chi$ and noting that $\chi = 0$ in $C_{\omega} \setminus D$ we get
\begin{equation*}
  E_0 \chi^2 u = \chi \tilde u_R + \chi \sum_{m=1}^{M'} c_m \varphi_m(\theta) r^{1 + z'_m} + \chi \sum_{m=1}^{M''}  [ c'_m \phi_m(\theta) +  c''_m\ln r \varphi_m(\theta)] r^{1 + z''_m},  
\end{equation*}
To obtain the expansion \eqref{eq_sing_expansion_2new}, firstly note that in $D$, we have  $u = E_0 \chi^2 u + (1-\chi^2) u$
and that the function $(1-\chi^2) u \in V^4_0(D)$. The latter comes from $H^4(D)$ elliptic regularity of $u$ and
since $1-\chi^2$ is supported away from the corner. We set $u_R = \chi \tilde u_R + (1-\chi^2) u$.
\end{proof}

\noindent
We now define the modification of $\eta_1$ that gives us a suitable solution to the Dirichlet problem 
for the biharmonic operator on $D$. More specifically we define $ \zeta_1$  as 
\begin{equation} \label{eq_zeta_1}
\begin{aligned}
 \zeta_1 := \chi \eta_1 + h = \chi r^{1-z_1} \varphi_1(\theta) + h , \qquad (r,\theta) \in D,
\end{aligned}
\end{equation}
where $\eta_1$ is given by \eqref{eq_eta_1} and $h$ solves the below. Recall
that the roots $z_m'$ and $z_m''$, and the functions $\varphi_m$ are universal
and do not depend on the source terms of the equations involved. The correction term $h$ solves
\begin{equation*} 
\begin{aligned}
\begin{cases}
\Delta^2 h =  [\chi,\Delta^2]  (r^{1-z_1} \varphi_1 ), \qquad \text{ in } D,\\
\p^j_\nu h |_{\p D} = 0, \qquad\qquad\quad\qquad  j=0,1.
\end{cases}
\end{aligned}
\end{equation*}
Notice that $h \in H^2_0(D)$, so that by Lemma \ref{lem_H20_in_V20} $h \in V^2_0(D)$. By
Proposition \ref{prop_singular_expansion_D} we can write $h$ in the form
\begin{equation} \label{eq_h_expansion}
\begin{aligned}
h = h_R + 
\chi \sum_{m=1}^{M'}  c_m \varphi_m(\theta) r^{1 + z'_m} 
+ 
\chi \sum_{m=1}^{M''}  [  c'_m \phi_m(\theta) +   c''_m\ln r \varphi_m(\theta)] r^{1 + z''_m},
\end{aligned}
\end{equation}
where $h_R \in V^4_0(D)$.
Notice also that
the definition of $\zeta_1$ shows that it is  biharmonic in $D$, which is what we wanted. 
More specifically 
\begin{equation} \label{eq_Delta2_eta1_0_2}
\begin{aligned}
\begin{cases}
\Delta^2 \zeta_1 (x) &= 0, \qquad \text{ for } x 	\in D. \\
\p^j_\nu \zeta_1 |_{\p D} &=0. 
\end{cases}
\end{aligned}
\end{equation}
Note also that we always have that 
\begin{align} \label{eq_zeta1_not_0}
\zeta_1 \neq 0,    
\end{align}
because $\zeta_1 = \chi \eta_1 + h$, and since $\eta_1 \neq 0$. Note that $\chi \eta_1$
and $h$ cannot cancel each other near the corner, since $\chi \eta_1 \sim r^{1-z_1}$ and
$h$ behaves according to \eqref{eq_h_expansion}.

We now derive the analogue of Lemma \ref{lem_biharmonic_c1_not_zero} for the bounded domain $D$. Its
proof is similar to the proof of Lemma \ref{lem_biharmonic_c1_not_zero}, with the main difference
that there are new additional terms due to the correction term $h$.

\begin{proposition} \label{prop_sing_cond_D}
Let  $\omega \in (\omega_0, 2\pi)$ and let
$u \in V^2_0(D)$ solve the equation
\begin{equation} \label{eq_biharmonic_with_source_D}
\begin{aligned}
\begin{cases}
\Delta^2 u = f, \qquad \text{ in } D,\\
\p^j_\nu u |_{\p D} = 0, \quad j=0,1,
\end{cases}
\end{aligned}
\end{equation}
with $f \in V^0_0(D)$, and suppose
$$
(f, \zeta_1)_{L^2(D)} \neq 0,
$$
where $ \zeta_1$ is given by \eqref{eq_zeta_1}. Then we have that 
$
u 	\sim  c_1 r^{1+z_1}\varphi_1(\theta),  
$
where $ c_1 \neq 0$.
\end{proposition}

\begin{proof}
The proof of the statement is essentially the same as for Lemma \ref{lem_biharmonic_c1_not_zero}.
We need to show that $ c_1 \not = 0$. We will do this by showing  that $ c_1$ can be expressed in terms of  
the inner product $(f, \zeta_1)_{L^2(D)}$.
Let $B_\eps := B(x_0, \eps)$, $\eps > 0$. Clearly
\begin{align} \label{eq_f_eta1_limnew}
(f, \zeta_1)_{L^2( D)} = 
\lim_{\eps \to 0}(f, \zeta_1)_{L^2( D \setminus B_\eps)}.
\end{align}
Set $D_\eps := D \setminus B_\eps$ and $\Gamma_\eps := D \cap \{ r=\eps  \}$.
Keeping in mind that $\zeta_1$ solves the equation \eqref{eq_Delta2_eta1_0_2} in $D_\eps$, 
we have by integrating by parts, similar as in the proof of Lemma \ref{lem_biharmonic_c1_not_zero}, that 
\begin{equation} \label{eq_as_bndry_ints_2}
\begin{aligned}
(f, \zeta_1)_{L^2( D_\eps )}
&=
(\Delta^2 u , \zeta_1)_{L^2( D_\eps )} \\
&= 
(\p_r \Delta u , \zeta_1)_{L^2(\Gamma_\eps )}
- (\Delta u ,\p_r \zeta_1)_{L^2(\Gamma_\eps )} \\
&+ (\p_r u ,\Delta  \zeta_1)_{L^2(\Gamma_\eps )}
- (u ,\p_r \Delta \zeta_1)_{L^2(\Gamma_\eps )}.
\end{aligned}
\end{equation}
Using the expansion \eqref{eq_sing_expansion_2new}, we can write
\begin{equation} \label{eq_u_expansion} 
\begin{aligned}
u(r,\theta) =  
\chi \sum_{m=1}^{M'}  c_m \varphi_m(\theta) r^{1 + z'_m} 
+ 
\chi \sum_{m=1}^{M''} [ c'_m \phi_m(\theta) +  c''_m \ln r \varphi_m(\theta)] r^{1 + z''_m} + u_R,
\end{aligned}
\end{equation}
where $u_R \in V_0^{4}(\mathcal C_\omega)$, and $z'_m$ and $z''_m$ are simple and double roots of equation
\eqref{eq_char_eq_2}. 

Using the definition of $\zeta_1 = \chi \eta_1 + h$, we can split the boundary terms in 
\eqref{eq_as_bndry_ints_2} into an $h$ part and a $\chi \eta_1$ part.
We now claim that the boundary terms from $h$ are such that  
\begin{equation} \label{eq_h_ints}
\begin{aligned}
(\p_r& \Delta u , h)_{L^2(\Gamma_\eps )}
- (\Delta u ,\p_r h)_{L^2(\Gamma_\eps )} \\
&+ (\p_r u ,\Delta  h)_{L^2(\Gamma_\eps )}
- (u ,\p_r \Delta  h)_{L^2(\Gamma_\eps )} \to 0, \quad \text{ as } r= \eps \to 0. 
\end{aligned}
\end{equation}
Let us evaluate these in the limit $r=\eps \to 0$. 
Since $u$ has the expansion \eqref{eq_u_expansion}, and $h$ the expansion \eqref{eq_h_expansion},
we can split the integrals in \eqref{eq_h_ints} into products of the individual terms in the expansions.
Let us begin with  analyzing the case with least decay in $r$, when $r \to 0$.
This is given by the $r^{1+z_1}$ terms in the expansions \eqref{eq_u_expansion} and \eqref{eq_h_expansion}, since
we set $z'_1 = z_1$, and $\Re z_1 < \Re z'_{m>1},\Re z''_{m}$.
Note that the behaviour of this type boundary term 
has in general faster decay than the corresponding boundary term in \eqref{eq_as_bndry_ints}, 
since here the terms behaves as $r^{1+z_1}$, but  \eqref{eq_as_bndry_ints} considers $\eta_1$ which behaves
$r^{1-z_1}$. It is thus to be expected that these boundary terms vanish.
The first of these boundary terms contributing to \eqref{eq_h_ints} is
\begin{align*}
(  \p_r\Delta r^{1+z_1} \varphi_1 , r^{1+z_1} \varphi_1 )_{L^2(\Gamma_\eps )} 
&=
  \int_0^\omega \p_r \Delta (r^{1+z_1} \varphi_1) r^{1+z_1} \varphi_1 r \, d\theta  \\
&=
\int_0^\omega r^{z_1-1} (-1+z_1)((1+z_1)^2\varphi_1 + \varphi''_1) r^{1+z_1} \varphi_1  \, d\theta  \\
&= C r^{2z_1} \to 0
\end{align*}
as $r = \eps \to 0$. Note that the fourth boundary term in \eqref{eq_h_ints} is of the same form.
The second and the third terms in \eqref{eq_h_ints} are of the form
\begin{align*}
(  \Delta r^{1+z_1} \varphi_1 , \p_r r^{1+z_1} \varphi_1 )_{L^2(\Gamma_\eps )} 
&=
\int_0^\omega \Delta (r^{1+z_1} \varphi_1) \p_r (r^{1+z_1} \varphi_1) r \, d\theta \\
&=
  \int_0^\omega r^{z_1-1} ((1+z_1)^2 \varphi_1 + \varphi_1'') (1+z_1) r^{1+z_1} \varphi_1 \, d\theta  \\
&\leq C r^{2z_1} \to 0
\end{align*}
as $r = \eps \to 0$. 
The contribution to \eqref{eq_h_ints} by the other products of the individual terms in the expansions
\eqref{eq_u_expansion} and \eqref{eq_h_expansion} can be estimated similarly. This shows that
that \eqref{eq_h_ints} holds.

Now we turn to the contribution of the $\chi\eta_1$ part of $\zeta_1$ 
to the terms in \eqref{eq_as_bndry_ints_2}. We argue as follows.
As $\chi = 1$ on $\Gamma_\eps$, when $\eps>0$ is small, we see that the boundary terms
in \eqref{eq_as_bndry_ints_2} due to the $\chi \eta_1$ part of $\zeta_1$ are identical to 
those in \eqref{eq_as_bndry_ints}, and we thus
get by identical computations as in the proof of Lemma \ref{lem_biharmonic_c1_not_zero}
that
\begin{align*}
(f,  \chi \eta_1)_{L^2(D)}
&=
 c_1 \Big(
4z_1(z_1^2-1) \| \varphi_1 \|^2_{L^ 2(0,\omega)}  - 4z_1\| \varphi'_1\|^ 2_{L^ 2(0,\omega)} 
\Big).
\end{align*}
Note that the term multiplying $ c_1$ on the right is less than zero, since $z_1 \in (0,1)$.
It follows that $c_1 \neq 0$, if $(f,  \zeta_1)_{L^2(D)} \neq 0$.
\end{proof}

\section{Localization at a  non-convex corner } \label{sec_non_convex_corner}

\noindent
In the following sections we will construct examples of ITEFs that localize at a non-convex corner.
We shall consider the following form of the interior transmission problem
\begin{equation} \label{eq_ite_1}
\begin{aligned}
\begin{cases}
(\Delta + k_1^2) v = 0, \quad \text{ in }\quad D, \\
(\Delta + k_2^2) w = 0, \quad \text{ in }\quad D, \\
\p_\nu^j(v-w)|_{\p D} = 0, \; j=0,1
\end{cases}
\end{aligned}
\end{equation}
where $k_1 \neq k_2$ are positive constants with $v-w \in H^2_0(D)$. 
We will study this problem by writing it as an equivalent 4th order equation. We will then, in the following
sections, find a number of solutions to the fourth order equation by studying a related eigenvalue problem. These solutions will give solutions to \eqref{eq_ite_1} that have a singularity at the corner, so will not be in $H^2(D)$ or other spaces that embed into continuous functions.

Suppose $(v,w)$ solves \eqref{eq_ite_1}. Then $u = v-w \in H^2_0(D)$ solves the 4th order equation below weakly.
\begin{equation} \label{eq_4th_order_problem}
\begin{aligned}
\begin{cases}
(\Delta + k_1^2) (\Delta + k_2^2) u = 0, \quad \text{ in }\quad D  \\
\p_\nu^j u |_{\p D} = 0, \; j=0,1
\end{cases}
\end{aligned}
\end{equation}
On the other hand suppose now that a solution $u$ of \eqref{eq_4th_order_problem} is given.
In this case we can define 
\begin{align*}
\tilde v := \frac{ 1 }{ k^2_2 - k^2_1 } (\Delta + k_2^2) u,
\qquad 
\tilde w := \frac{ 1 }{ k^2_2 - k^2_1 } (\Delta + k_1^2) u,
\end{align*}
It is clear that $(\tilde v, \tilde w)$ solves two first equations in
\eqref{eq_ite_1}. The pair $(\tilde v, \tilde w)$ also satisfies the
boundary- and $H^2$-integrability conditions since
$$
\tilde v - \tilde w = \frac{ 1 }{ k^2_2 - k^2_1 } (k_2^2 - k^2_1) u = u,
$$
and thus $\tilde v - \tilde w \in H^2_0(D)$. We thus see that $(\tilde v ,
\tilde w)$ is a solution to \eqref{eq_ite_1}.  We can thus obtain a solution to
\eqref{eq_ite_1} by solving \eqref{eq_4th_order_problem}.

\subsection{A buckling type problem}\label{sec_buckling}

Here we study a 4th order eigenvalue problem that resembles the problem of a buckling 
plate. Our goal is to obtain a basis of eigenfunctions and eigenvalues 
that we can convert to solutions  of \eqref{eq_4th_order_problem}.

\medskip
\noindent
The idea is to look at the buckling type problem
\begin{align} \label{eq_buckling_2}
\begin{cases}
\Delta(\Delta + \tilde \kappa) u &= \tilde \lambda u , \quad    \text{ on } \quad D, \\
\qquad \p^j_\nu u |_{\p D} &= 0,\qquad j=0,1,
\end{cases}
\end{align}
where $ \tilde \kappa \in \R \setminus \{ 0 \}$ and $\tilde \lambda \in \R$ is traditionally a spectral parameter. We will find solutions to this 
problem, which will then in turn give solutions to the 4th order problem \eqref{eq_4th_order_problem},
which gives us solutions to \eqref{eq_ite_1}.

To relate the solution $u$ of \eqref{eq_4th_order_problem}, to \eqref{eq_buckling_2}, we write
the equation for $u$ as
\begin{equation} \label{eq_4th_order_alternate}
\begin{aligned}
\Delta^2 u + (k_1^2+k^2_2)\Delta u + k^2_1k^2_2 u = 0.
\end{aligned}
\end{equation}
Let's find out how the $\tilde \kappa$ and $\tilde \lambda$ in \eqref{eq_buckling_2}
translate into some $k_1$ and $k_2$ when we rewrite the equation \eqref{eq_buckling_2}
in the form of \eqref{eq_4th_order_alternate}.
It will be convenient to set $a := k^2_1$ and $b := k^2_2$. 
To solve for $a$ and $b$, when $\tilde \kappa$ and $\tilde \lambda$ are given, we use
$$
\tilde \lambda = -ab \quad\text{and} \quad   \tilde \kappa = a+b.
$$
This gives that 
$$
a - \frac{ \tilde \lambda }{ a } -  \tilde \kappa = 0 
\quad \Leftrightarrow \quad
a^2 -  \tilde \kappa a - \tilde \lambda  =0
\quad \Leftrightarrow \quad
a = \frac{ 1 }{ 2 }(  \tilde \kappa \pm \sqrt{ \tilde \kappa^2 + 4\tilde \lambda}) 
$$
Now choose 
$$
a = \frac{ 1 }{ 2 }( \tilde  \kappa + \sqrt{ \tilde \kappa^2 + 4\tilde \lambda}) 
\qquad
b = \frac{ 1 }{ 2 }( \tilde  \kappa - \sqrt{ \tilde \kappa^2 + 4\tilde \lambda}).
$$
This gives
\begin{equation} \label{eq_getting_ks}
\begin{aligned}
k_1 = \pm \sqrt{\frac{ 1 }{ 2 }( \tilde  \kappa + \sqrt{ \tilde \kappa^2 + 4\tilde \lambda}) }
\qquad
k_2 =  \pm \sqrt{\frac{ 1 }{ 2 }( \tilde  \kappa - \sqrt{ \tilde \kappa^2 + 4\tilde \lambda})}.
\end{aligned}
\end{equation}
This shows that if we solve \eqref{eq_buckling_2} we get solutions to \eqref{eq_4th_order_problem}
and \eqref{eq_ite_1}, with $k_1$ and $k_2$ given by \eqref{eq_getting_ks}.
Note that in order to obtain real ITEVs we need that
\begin{equation} \label{eq_real_itev_cond}
\begin{aligned}
 0 < \tilde \kappa^2 + 4\tilde \lambda < \tilde \kappa^2
\quad \Leftrightarrow \quad
\tilde \lambda \in ( -\tilde \kappa^2/4, 0 ).
\end{aligned}
\end{equation}

\subsection{An auxiliary eigenvalue problem} \label{sec_auxiliary}
We formulate an auxiliary eigenvalue problem which will gives us a number of solutions
to the buckling type problem \eqref{eq_buckling_2}.
More specifically we are interested in the eigenvalue problem
$$
K_{\kappa,\lambda} \psi_j = \sigma_j \psi_j,
$$
where $\kappa <0 $, $\lambda > 0$, 
and where
$$
K_{\kappa, \lambda} := \Delta^{-2} (\kappa \Delta - \lambda) : H^2_0(D) \to H^2_0(D),  
$$
where $\Delta^{-2}: H^{-1}(D) \to H^{2}_0(D)$ is the solution operator
to \eqref{eq_biharmonic_with_source_D}, given by Lemma 7.1.1 in \cite{Gr85}. 
 
\begin{lemma} \label{lem_K_compact}
The operator $K_{\kappa,\lambda}: H^2_0(D) \to H^2_0(D)$ is compact.
\end{lemma}

\begin{proof}
By Corollary 3.4.3 in \cite{Gr92} we have that $\Delta^{-2}:H^{-1}(D) \to H^{5/2+\eps}(D) \cap H^2_0(D)$, where $\eps > 0$.
The compactness of $K_{\kappa,\lambda}$ follows from the inclusion operator
$$id : H^{5/2+\eps}(D) \to H^2(D)$$
being compact. See
e.g. Theorem 3.27 p. 87 in \cite{Mc00}.
\end{proof}

\noindent
Next, we will define a useful and equivalent inner product on the Hilbert space $H^2_0(D)$.
We set
$$
\langle \psi , \phi  \rangle_{\Delta^2} := \int_D \Delta \psi \Delta \phi \,dx.
$$
See also \cite{BL13}.
The following Lemma will show that the operator $K_{\kappa,\lambda}$ is positive on $H^{2}_0(D)$
for suitable parameters $\lambda$ and $\kappa$.

\begin{lemma} \label{lem_K_positive}
The operator $K_{\kappa,\lambda}$ is positive on $H^{2}_0(D)$, so that
$$
\langle K_{\kappa,\lambda} \psi , \psi  \rangle_{H^2} \geq 0, \qquad \forall \psi \in H^2_0(D),
$$
when $-\kappa / \lambda \geq \lambda_1^{-1}$, where $\lambda_1$ is the first Dirichlet eigenvalue of the Laplacian
on $D$.
\end{lemma}

\begin{proof}
Let $\psi \in H^2_0(D)$. We utilize the alternative inner product and write
\begin{align*}
\langle K_{\kappa,\lambda} \psi , \psi  \rangle_{\Delta^2} 
&= \int_D \Delta \Delta^{-2} (\kappa\Delta - \lambda) \psi \Delta \psi \,dx\\
&= \int_D \Delta^2 \Delta^{-2} (\kappa\Delta - \lambda) \psi  \psi \,dx\\
&= \int_D (\kappa\Delta - \lambda) \psi  \psi \,dx\\
&= -\kappa \int_D |\nabla \psi|^2 \, dx  - \lambda \int_D \psi^2  \,dx.
\end{align*}
Recall that the Poincaré inequality says that 
$$
\| \psi \|^2_{L^2(D)} \leq \lambda_1^{-1} \| \nabla \psi \|^2_{L^2(D)},
$$
where $\lambda_1$ is the first Dirichlet eigenvalue of the Laplacian.
We thus obtain that  
\begin{align*}
\langle K_{\kappa,\lambda} \psi , \psi  \rangle_{\Delta^2} 
\geq 
(-\kappa - \lambda \lambda^{-1}_1) \| \nabla \psi \|^2_{L^2(D)}.
\end{align*}
The operator $K_{\kappa,\lambda}$ is thus positive if 
\begin{align} \label{eq_K_pos_cond}
-\kappa - \lambda \lambda^{-1}_1 \geq 0 
\quad \Leftrightarrow \quad
\frac{ -\kappa }{\lambda  } \geq \lambda_1^{-1}.
\end{align}
\end{proof}

\noindent
It is also easy to see that $K_{\kappa, \lambda}$ is self-adjoint as the next lemma shows.

\begin{lemma} \label{lem_K_self_adjoint}
The operator $K_{\kappa, \lambda} : H^2_0(D) \to H^2_0(D)$ is self-adjoint.  
\end{lemma}

\begin{proof}
We will show that $K_{\kappa, \lambda}$ is symmetric and bounded, from which the claim follows.
We have that
\begin{align*}
\langle K_{\kappa,\lambda} \psi , \phi  \rangle_{\Delta^2} 
&= \int_D \Delta \Delta^{-2} (\kappa\Delta - \lambda) \psi \Delta \phi \,dx\\
&= \int_D  (\kappa\Delta - \lambda) \psi  \phi \,dx\\
&= \int_D  \psi (\kappa\Delta - \lambda) \phi \,dx\\
&= \int_D  \phi \Delta \Delta^{-2} (\kappa\Delta - \lambda) \psi \,dx\\
&= \langle \psi, K_{\kappa,\lambda} \phi   \rangle_{\Delta^2} 
\end{align*}
The operator $K_{\kappa,\lambda}$ is bounded, since $(\Delta -\lambda): H^2_0(D) \to L^2(D)$
is bounded, and $\Delta^{-2} : L^2(D) \to H^4(D)$ is bounded.
\end{proof}

\noindent
Lemmas \ref{lem_K_compact}, \ref{lem_K_positive} and \ref{lem_K_self_adjoint} and the 
basic spectral theory of compact operators now gives the following characterization 
of the spectrum and eigenfunctions of $K_{\kappa, \lambda}$. See
e.g. Theorem VI.16 in \cite{RS72}.

\begin{proposition} \label{prop_K_spectrum}
$K_{\kappa,\lambda}: H^2_0(D) \to H^2_0(D)$ is a compact self adjoint operator that is positive, 
it has a spectrum consisting of eigenvalues
$$
\sigma_j > 0, \qquad \sigma_1 \geq \sigma_2 \geq \dots \geq \sigma_j \to 0,
$$
as $j \to \infty$, and corresponding eigenfunctions $\psi_j \in H^2_0(D)$ solving
$$
K_{\kappa,\lambda} \psi_j = \sigma_j \psi_j.
$$
The eigenfunctions $\{ \psi_j\}$ form a complete orthonormal basis of $H^2_0(D)$.
\end{proposition}

\noindent
Note that the point of looking at the operator $K_{\kappa, \lambda}$ is that the eigenfunctions
solve also a version of problem \eqref{eq_buckling_2}. More specifically
the $\psi_j$ solve
\begin{equation} \label{eq_other_ev_problem}
\begin{aligned}
\Delta ( \Delta  - \tfrac{\kappa}{\sigma_j}) \psi_j = - \tfrac{\lambda}{\sigma_j} \psi_j.
\end{aligned}
\end{equation}
We thus get solutions to \eqref{eq_buckling_2} by setting 
\begin{equation} \label{eq_tilde_translation}
\begin{aligned}
\tilde \kappa_j := - \tfrac{\kappa}{\sigma_j}, \qquad
\tilde \lambda_j := - \tfrac{\lambda}{\sigma_j}. 
\end{aligned}
\end{equation}
Note that \eqref{eq_other_ev_problem} should be interpreted weakly, since $\psi_j \in H^2_0(D)$. And therefore 
\eqref{eq_other_ev_problem} means that
\begin{equation} \label{eq_other_ev_problem_weak}
\begin{aligned}
\int_D \Delta \psi_j \Delta \varphi + \tilde \kappa_j \nabla \psi_j \cdot \nabla \varphi \,dx =  - \tilde \lambda_j \int_D \psi_j\varphi\,dx
\qquad \forall \varphi \in H^2_0(D).
\end{aligned}
\end{equation}

Recall that if we want real ITEVs we need condition \eqref{eq_real_itev_cond} to be satisfied.
We thus want that 
$$
\tilde \lambda_j \in \Big ( \frac{ - \tilde \kappa_j^2 }{ 4  }, 0 \Big)
$$
This gives the condition
\begin{equation} \label{eq_itev_cond}
\begin{aligned}
- \tfrac{\lambda}{\sigma_j}  \in \Big ( -\frac{ \kappa^2 }{ 4 \sigma_j^2 }, \,0 \Big)
\quad \Leftrightarrow \quad
- \lambda  \in \Big ( -\frac{ \kappa^2 }{ 4 \sigma_j }, \,0 \Big),
\end{aligned}
\end{equation}
which clearly holds for any fixed $\kappa < 0$ and $\lambda > 0$ when $\sigma_j$ becomes small.

\medskip
\noindent
Let us finally show that  we can choose the parameters $\lambda > 0$ and $\kappa < 0$,
so that all the ITEVs given by \eqref{eq_getting_ks} from the $(\tilde \kappa_j, \tilde \lambda_j)$, $j=1,2,3,\dots$ are real.
We first recall the definition of the spectral radius of $K_{\kappa, \lambda}$. This is defined as
\begin{equation} \label{eq_spec_radius}
\begin{aligned}
r(K_{\kappa, \lambda}) := \sup \{ |\sigma_j| \}.
\end{aligned}
\end{equation}
Since $H^2_0(D)$ is a Hilbert space we have that
\begin{equation} \label{eq_gelfand_formula}
\begin{aligned}
r(K_{\kappa, \lambda}) = \| K_{\kappa, \lambda} \|_{H^2_0(D) \to H^2_0(D)}.
\end{aligned}
\end{equation}
See Theorem VI.6 in \cite{RS72}.

\begin{proposition} \label{prop_real_ITEVs}
Set $\lambda = 1$, and assume  that $-\kappa>0$ is large, then all ITEVs given
by the $(\tilde \kappa_j, \tilde \lambda_j)$
in \eqref{eq_tilde_translation}, by means of \eqref{eq_getting_ks}, are real.
\end{proposition}

\begin{proof}
Fix $\lambda = 1 $. 
Let's first check that the condition for the positivity of $K_{\kappa, \lambda}$ in \eqref{eq_K_pos_cond} can be satisfied. 
The condition in \eqref{eq_K_pos_cond} is 
$$
-\kappa \geq \lambda_1^{-1},
$$
which clearly holds when $|\kappa|$ is large.

Next we want to show that the condition in \eqref{eq_itev_cond}, which guarantees that the ITEVs will be real, 
can be satisfied. We need to show that
\begin{equation} \label{eq_realness_lambda_1}
\begin{aligned}
-1  \in \Big ( -\frac{ \kappa^2 }{ 4 \sigma_j }, \,0 \Big).
\end{aligned}
\end{equation}
for large $-\kappa>0$. We can estimate the size of $\sigma_j$ in terms of the parameter $\kappa$.
From \eqref{eq_gelfand_formula} we have that
$$
\sigma_j \leq r(K_{\kappa, 1 }) = \| K_{\kappa, 1} \|_{H^2_0(D) \to H^2_0(D)}.
$$
We have the mapping property $\Delta^{-2}:H^{-1}(D) \to H^{5/2+\eps}(D) \cap H^2_0(D)$ for any $\eps > 0$ by Corollary~3.4.3 in \cite{Gr92}.
Thus
\begin{align*}
\| K_{\kappa, 1} u \|_{H^2(D)} 
= \big \| \Delta^{-2} (\kappa \Delta - 1)  u \big \|_{H^2(D)} 
\leq C \big\| (\kappa \Delta - 1)  u \big \|_{L^2(D)}
\leq C |\kappa|\big \| u \big \|_{H^2(D)},
\end{align*}
when $-\kappa > 1$. By combining the above, we get
$$
\sigma_j \leq \| K_{\kappa, 1} \|_{H^2_0(D) \to H^2_0(D)} \leq C |\kappa|.
$$
For large $-\kappa > 0$, we have hence that
$$
\frac{ \kappa^2 }{ 4 \sigma_j } \geq  \frac{ \kappa^2 }{ 4C |\kappa| } \geq \frac{ |\kappa| }{ 4C} > 1,  
$$
and therefore \eqref{eq_realness_lambda_1} holds.
\end{proof}

\subsection{Finding a singular ITEF } \label{sec_sing_ITEF}
In this subsection we derive the existence of ITEFs that localize at a corner
using the conditions for the appearance of a singularity, which we derived in section \ref{sec_a_condition_for_sing}.

\medskip
\noindent
The following lemma shows that since we have many ITEFs $\psi_j$ we can find one that has $\Delta \psi_j$ singular.
We do this using the completeness of the set $\{\psi_j\}$ and the condition of Proposition \ref{prop_sing_cond_D}.

\begin{lemma} \label{lem_f_sing}
Let $\omega \in (\omega_0,2\pi)$ and let the eigenfunctions $\psi_j$,
$j=1,2,\dots$, of $K_{\kappa,\lambda}$ be given by Proposition
\ref{prop_K_spectrum}.
Then there exists $j \in \N$ and $\alpha \in (0,1)$ such that
$$
\Delta \psi_j \sim C(\theta) r^{-\alpha} \qquad \mbox{ as } \quad r \to 0,
$$
where $C$ is not identically zero.
\end{lemma}

\begin{proof}
Let $z'_1$ be the simple root given by Lemma \ref{lem_z1_real} and let the corresponding
$\zeta_1$ be as in Proposition \ref{prop_sing_cond_D}. 
Let $w \in H^2_0(D)$ be the unique solution (see e.g. Lemma 7.1.1 in \cite{Gr85}) of 
\begin{align} \label{eq_f_bvp}
\begin{cases}
\Delta^2 w &= - \lambda \zeta_1 + \kappa \Delta \zeta_1  , \quad    \text{ on } \quad D, \\
 \p^j_\nu w |_{\p D} &= 0,\qquad j=0,1,
\end{cases}
\end{align}
where $\zeta_1$ is given by \eqref{eq_zeta_1}.
Note that the source term in the above equation is non-zero due to \eqref{eq_zeta1_not_0}.
If $- \lambda \zeta_1 + \kappa \Delta \zeta_1 = 0$, then as $\p^j_\nu \zeta_1 = 0$ on
$\p D$, the unique continuation principle for the Laplacian would imply that $\zeta_1 \equiv 0$,
in contradiction with \eqref{eq_zeta1_not_0}.

The functions $\{ \psi_j\}$ form an orthogonal basis of $H^2_0(D)$, and hence we may write
$$
w = \sum_j b_j \psi_j \qquad b_j := \langle w , \psi_j  \rangle_{\Delta^2}. 
$$
Now we consider the coefficients $b_j$ and use the definition of a weak solution
to write
\begin{equation} \label{eq_bj}
\begin{aligned}
b_j 
= \int_D \Delta w \Delta \psi_j \, dx 
= \int_D  \psi_j (-\lambda \zeta_1 + \kappa \Delta \zeta_1   ) \, dx.
\end{aligned}
\end{equation}
The coefficients $b_j$ cannot be zero 
for all $j=1,2,3,\dots$, since then $w = 0$, which is a contradiction. 
By dividing by $-\sigma_j$ and switching to $\tilde \kappa_k, \tilde \lambda_j$ from \eqref{eq_tilde_translation}, we see that for at least one $j$ we have
$$
\int_D  \psi_j (\tilde \lambda_j \zeta_1 - \tilde \kappa_j \Delta \zeta_1   ) \, dx \neq 0.
$$
By integrating by parts and using the fact that $\psi_j \in H^2_0(D)$, we get
\begin{equation} \label{eq_f_cond_for_psi}
\begin{aligned}
\int_D  \zeta_1 (\tilde \lambda_j \psi_j - \tilde \kappa_j \Delta \psi_j ) \, dx \neq 0.
\end{aligned}
\end{equation}
In addition, since $\psi_j \in H^2_0(D)$, we have $\psi_j \in V^2_0(D)$
by an analogue of Lemma \ref{lem_H20_in_V20}.
Therefore, referring to \eqref{eq_other_ev_problem}, we see that $\psi_j \in V^2_0(D)$ solves 
\begin{align} \label{eq_f_bvp}
\begin{cases}
\Delta^2 \psi_j  &= \tilde \lambda_j \psi_j - \tilde \kappa_j \Delta \psi_j   , \quad    \text{ on } \quad D, \\
 \p^j_\nu \psi_j |_{\p D} &= 0,\qquad j=0,1.
\end{cases}
\end{align}
We can use Proposition \ref{prop_singular_expansion_D} to write
$$
\psi_j = \psi_{j,S} + \psi_{j,R}
$$
where $\psi_{j,R} \in V^4_0(D)$ and $\psi_{j,S}$ has the singular parts.
We see from equation \eqref{eq_f_cond_for_psi} that the condition in Proposition \ref{prop_sing_cond_D} 
is satisfied by the source in \eqref{eq_f_bvp}. Thus Proposition \ref{prop_sing_cond_D} implies
that the asymptotic term $\psi_{j,S}$  has a non-zero  $c_1  \chi r^{1+z'_1}\varphi(\theta)$ part, and that
\begin{equation} \label{eq_asym_term_non_zero}
\begin{aligned}
\psi_j = c_1 \chi r^{1 + z'_1}\varphi_1(\theta)  + \tilde \psi_{j,R}, \qquad z'_1 \in (0,1), \qquad c_1 \varphi_1 \not \equiv 0,
\end{aligned}
\end{equation}
where $\tilde \psi_{j,R}$ contains the better behaved singular terms and the regular part in the expansion \eqref{eq_sing_expansion_2new}
of $\psi_j$, i.e.
$$
\tilde \psi_{j,R} = 
 \chi \sum_{m=2}^{M'} c_m\varphi_m(\theta) r^{1 + z'_m}  + 
 \chi \sum_{m=1}^{M''} [c'_m\phi_m(\theta) + c''_m \ln r \varphi_m(\theta)] r^{1 + z''_m} + \psi_{j,R}.
$$
Computing $\Delta \psi_j$ from this and \eqref{eq_asym_term_non_zero}  shows that 
$$
\Delta \psi_j \sim \Delta [c_1 r^{1 + z'_1}\varphi_1(\theta)] = C(\theta) r^{z'_1 - 1} \quad \mbox{ as } \quad r \to 0.
$$
Note that this is singular at $r=0$ since $z'_1 \in (0,1)$, provided that $C(\theta) \neq 0$.
The $C(\theta)$ cannot however be zero, since otherwise $\Delta [C r^{1 + z'_1}\varphi_1(\theta)]=0$
with the boundary conditions $ \p^k_\nu[C r^{1 + z'_1}\varphi_1(\theta)]|_{\p \mathcal C_\omega \cap D}=0$, $k=0,1$.
The unique continuation principle would then imply that $C(\theta) r^{1 + z'_1}\varphi_1(\theta) \equiv 0$
near the corner point,
in contradiction with \eqref{eq_asym_term_non_zero}.
\end{proof}

\noindent
Lemma \ref{lem_f_sing} gives us a singular $\Delta \psi_j$ and we can use this to find a singular ITEF.
This goes as follows. The previous Lemma shows that
\begin{equation} \label{eq_Delta_sing}
\begin{aligned}
\Delta \psi_j \sim C(\theta) r^{-\alpha},  
\end{aligned}
\end{equation}
for some $\alpha \in (0,1)$, with $C(\theta) \not \equiv 0$.
Moreover if we consider the $\psi_j$ and $\sigma_j$, given by a large
$-\kappa>0$ and $\lambda = 1$ which is not a Dirichlet eigenvalue,  Proposition \ref{prop_real_ITEVs}
gives us
real ITEVs with $k_1,k_2$ built from $\kappa,\lambda,\sigma_j$ using the formulas \eqref{eq_getting_ks} and \eqref{eq_tilde_translation}. Note that we can pick $k_1,k_2>0$ in formula \eqref{eq_getting_ks}, which together
with \eqref{eq_real_itev_cond} implies that $k_1 \neq k_2$.
In more details, to build solutions to \eqref{eq_ite_1}, we  can define  $v$ and $w$ as
\begin{equation} \label{eq_v_w_def}
\begin{aligned}
v := \frac{ 1 }{ k^2_2 -\omega^2_1 } (\Delta + k_2^2) \psi_j,
\qquad 
w := \frac{ 1 }{ k^2_2 -\omega^2_1 } (\Delta + k_1^2) \psi_j.
\end{aligned}
\end{equation}
Since $\psi_j$ solves \eqref{eq_other_ev_problem} with $k_1^2+k_2^2 = \tilde \kappa_j$ and $k_1^2 k_2^2 = - \tilde \lambda_j$ by \eqref{eq_getting_ks}, we see that
$$
\Delta^2 \psi_j + (k_1^2+k^2_2)\Delta \psi_j  + k^2_1 k^2_2 \psi_j  = 0.
$$
Thus we have a solution to the problem
\begin{equation*} 
\begin{aligned}
\begin{cases}
(\Delta + k_1^2) v = 0, \quad \text{ in }\quad D, \\
(\Delta + k_2^2) w = 0, \quad \text{ in }\quad D, \\
v-w \in H^2_0(D)
\end{cases}
\end{aligned}
\end{equation*}
with $k_1^2 \neq k_2^2$. 
Now  $\psi_j \in H_0^2(D) \subset L^\infty(D)$, and hence we see from \eqref{eq_v_w_def} and \eqref{eq_Delta_sing}
that $w$ and $v$ are singular, so that
$$
v,w \sim C(\theta) r^{-\alpha}, 
$$
for some $\alpha \in (0,1)$, 
provided that  $\omega \in (\omega_0, 2 \pi)$.

Let us summarize the above discussion in the following Proposition, from which Theorem \ref{thm_main_loc}
immediately follows. 

\begin{proposition} \label{prop_main_loc}
Let $\omega \in (\omega_0,2\pi)$.
Assume $\lambda = 1$
and $-\kappa>0$ is large, and let $\psi_j$ and $\sigma_j$ be given by Proposition \ref{prop_K_spectrum}.
Then the ITEVs to \eqref{eq_ite_1} obtained from $\tilde \kappa_j$ and $\tilde \lambda_j$ by means of \eqref{eq_getting_ks} are real and $k_1^2 \neq k_2^2$.
Furthermore the ITEFs $(v,w)$ obtained by \eqref{eq_v_w_def} are such that 
$$
v,w \sim C(\theta) r^{-\alpha}, 
$$
for some $\alpha \in (0,1)$, 
where $C \not \equiv 0$.
\end{proposition}

\section{Vanishing near  convex corners }\label{sec_vanish}

\noindent
We will now study the vanishing of ITEFs at a convex corner.
We consider the refractive index $n \in C^1(\overline D)$, $n > 0$, and the interior transmission problem
\begin{equation} \label{eq_ite_with_n}
\begin{aligned}
\begin{cases}
(\Delta + k^2) v = 0, \quad \text{ in } \quad D, \\
(\Delta + k^2n) w = 0, \quad \text{ in }\quad D, \\
\p^j_\nu (v-w)|_{\p D} = 0,
\end{cases}
\end{aligned}
\end{equation}
where $k >0$ is a  constant and $\nu \in \{0,1\}$. We will also assume that 
\begin{equation} \label{eq_n_corner_assumption}
\begin{aligned}
n(x_0) \neq 1.
\end{aligned}
\end{equation}
Notice also that the problem \eqref{eq_ite_with_n} becomes trivial if $n\equiv 1$.
We will now assume that $D \subset \R^2$ has a single corner at $x_0$ with an angle $\omega \in (0,\pi)$, and that
$\p D \setminus \{ x_0 \}$ is smooth. See figure \ref{fig_convex}

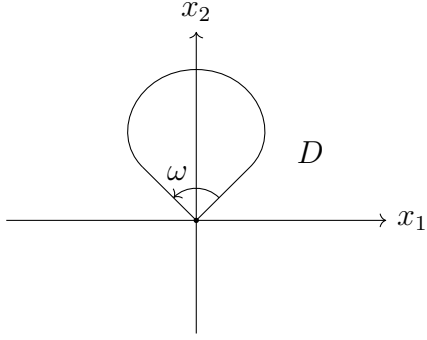
\begin{figure}[H]
  \begin{center}
    \begin{tikzpicture}
      \draw[->] (-2.5,0) -- (2.5,0) node[right] {$x_1$};
      \draw[->] (0,-1.5) -- (0,2.5) node[above] {$x_2$};
      \draw (0.7071,0.7071) to (0,0);
      \draw (-0.7071,0.7071) to (0,0);
      \draw (0.7071,0.7071) to[out angle=45, in angle=135, curve through = {(0,2)}] (-0.7071,0.7071); 
      \fill (0,0) circle (1pt);
      \draw (1.5,0.9) node {$D$};
      \draw[->] (0.3,0.3) arc (45:135:0.424) node[left, pos=.9, xshift=8pt, yshift=8pt] {$\omega$};
    \end{tikzpicture}
    \caption{An example domain $D$ with a convex corner of opening angle $\omega < \pi$. At this type of corner the ITEF vanishes.
			 The localization of an ITEF is thus not possible for convex corners. }
    \label{fig_convex}
  \end{center}
\end{figure}

\medskip
\noindent
From Proposition \ref{prop_V_Dirichlet} we can obtain the following result that we will use to show that ITEFs
for \eqref{eq_ite_with_n} and convex corners vanish at the corner point  (see also \cite{Da16} Theorem~L.1 for
an alternative derivation).

\begin{lemma} \label{lem_convex}
Assume that $\omega < \pi$ and that $f \in L^2(D)$. Suppose that $u \in H^1(D)$ solves
\begin{align} \label{eq_H1_source_problem}
\begin{cases}
\Delta u &= f , \quad \text{ in }\quad D\\
u|_{\p D} &= 0.
\end{cases}
\end{align}
Then we have $u \in H^2(D)$ and $r^{-2} u , r^{-1}\nabla u \in L^2(D)$.
\end{lemma}

\begin{proof} Firstly note that $f \in V^0_0(D)$ since the latter is equal to $L^2(D)$.
Now since $\omega < \pi$, we can set $\beta = 0$ in 
Proposition \ref{prop_V_Dirichlet} and obtain a unique solution $\tilde u \in V^2_0(D)$
to \eqref{eq_H1_source_problem}. We also have $V^2_0(D) \subset H^1(D)$.
Since $u$ is unique by the $H^1$-theory, we see that $u = \tilde u$.
From the definition of $V^2_0(D)$-norm, we obtain directly the facts that 
$u \in H^2(D)$ and $r^{-2} u , r^{-1}\nabla u \in L^2(D)$.
\end{proof}

\medskip
\noindent
The following Proposition shows that we can prove that the ITEFs $v$ and $w$ vanish at convex corner
without apriori assuming that $w$ and $v$ are bounded. Theorem \ref{thm_main_vanishing} is 
a direct consequence of this.

\begin{proposition} \label{prop_simple_vanish}
Suppose that $v,w \in H^1(D)$ are solutions to \eqref{eq_ite_with_n}
and that the opening angle $\omega < \pi$ for the corner point $x_0$. 
Then $w(x_0) = 0$ in the following sense:
\[
  \lim_{\varepsilon \to 0} \frac{1}{|\mathcal C_{\varepsilon}|} \int_{\mathcal C_\varepsilon} w(x) dx = 0,
\]
where $\mathcal C_{\varepsilon} = B(x_0,\varepsilon) \cap D$. The same applies also to $v$.
\end{proposition}

\begin{proof}
Consider the function $ u := v-w$ this solves the problem
\begin{align} \label{eq_u}
(\Delta + k^2) u = k^2(n - 1)w \quad \text{ in }\quad D.
\end{align}
Define $u_j := \p_j (v-w)$. Then  
\begin{align} \label{eq_u_j}
(\Delta + k^2) u_j = k^2 \p_j ((n-1)w) \quad \text{ in }\quad D.
\end{align}
This implies that $\Delta u_j \in L^2(D)$ because $w \in H^1(D)$ and $n \in C^1(\overline{D})$.
Let $\chi$ be a cut-off function as in the proof of Proposition \ref{prop_sing_cond_D} 
related to corner at $x_0=0$.  Then
\begin{align*}
(\Delta + k^2) (\chi u_j) = k^2 \chi \p_j \big( (n-1)w \big) + [\Delta, \chi] u_j =: f \quad \text{ in }\quad D.
\end{align*}
Notice firstly that $f \in L^2(D)$.
Utilizing  the vector identity 
$$
\nabla u = (\nu \cdot  \nabla u)\nu - \nu \times (\nu \times u)\quad \text{ on } \quad \p D,
$$
and the fact that $\p_\nu u |_{\p D} = 0$ and $u|_{\p D} = 0$ gives 
$$
\nabla (v-w) |_{\p D} = 0.
$$
Thus in particular, since $u_j = \mathbf e_j \cdot \nabla (v-w)$, we have
$$
\chi u_j |_{\p D} = 0.
$$
Hence we see that
\begin{equation} \label{eq_ur}
\begin{aligned}
\begin{cases}
(\Delta + k^2) (\chi u_j) =  f, \quad \text{ in }\quad D, \\
\phantom{(\Delta + k^2) [}\chi u_j |_{\p D} = 0.
\end{cases}
\end{aligned}
\end{equation}
From Lemma \ref{lem_convex}, we get the corner behaviors
$$
r^{-2} (\chi u_j), \, r^{-1}\nabla (\chi u_j) \in L^2(D).
$$
Restricting to $\mathcal C_\eps := B(x_0,\eps) \cap D$ where $\chi \equiv 1$, we see that 
$$
r^{-1} \p^2_{j} u \in L^2(\mathcal C_\eps), \qquad j=1,2, 
$$
and thus 
$$
r^{-1} \Delta u \in L^2(\mathcal C_\eps).
$$
We also have by definition that
$
u \in L^2(\mathcal C_\eps).
$
Now, using \eqref{eq_u} and \eqref{eq_n_corner_assumption}, we see that

\begin{equation}
  \label{eq:wrinL2}
  r^{-1} w = r^{-1} (k^2n - k^2)^{-1} (\Delta + k^2) u   \in L^2(\mathcal C_\eps)
\end{equation}
when $\eps$ is small enough so that $n\neq 1$ in $\mathcal C_{\eps}$.
Let $\varepsilon > 0$. Then by Cauchy--Schwarz
\begin{equation*} 
\begin{aligned}
\frac{1}{|\mathcal C_{\varepsilon}|} \int_{\mathcal C_{\varepsilon}} \left| w(x) \right| dx 
&= 
\frac{c}{\varepsilon^2} \int_{\mathcal C_{\varepsilon}} \left| w(x) r^{-1} \right| r dx \\ 
&\leq 
\frac{c}{\varepsilon^2} \sqrt{\int_{\mathcal C_{\varepsilon}} \left| w(x) r^{-1} \right|^2 dx} \sqrt{\int_{\mathcal C_{\varepsilon}} r^2 dx} \\
&= 
\frac{c}{\varepsilon^2} \lVert w r^{-1} \rVert_{L^2(\mathcal C_{\varepsilon})} 
\sqrt{c \int_0^{\varepsilon} r^3 dr } \leq c \lVert w r^{-1} \rVert_{L^2(\mathcal C_{\varepsilon})} \to 0
\end{aligned}
\end{equation*}
as $\varepsilon \to 0$. This is because of \eqref{eq:wrinL2} and so dominated convergence gives the zero limit.
\end{proof}

\section{Relation to corner scattering and open problems} \label{sec_corner_scattering}

\noindent
In this section we explore the connections of our results to the scattering of penetrable obstacles with corners. Let $V \in L^{\infty}(\R^n)$ be a real-valued function with $V = 1$ outside a bounded domain $D$. Inside $D$ we have $V = n$, where $n$ is the refractive index as in \eqref{eq_ite}. In scattering, one probes the medium modelled by $V$ with an incident wave $u^i$ which satisfies
\begin{equation}
  \label{eq:incidentwave}
  \Delta u^i + k^2 u^i = 0 \qquad \text{ in } \R^n.
\end{equation}
The presence of the inhomogeneity $V$ perturbs the propagation of $u^i$, so the physical total field $u$ satisfies
\begin{equation}
  \label{eq:totalwave}
  \Delta u + k^2 V u = 0 \qquad \text{ in } \R^n,
\end{equation}
and we assume that the total wave is a superposition of the incident wave and a scattered wave $u = u^i + u^s$. The direction of time, causality, or energy propagation to infinity is ensured by requiring the Sommerfeld radiation condition
\begin{equation}
  \label{eq:sommerfeld}
  \lim_{r \to \infty}r^{\frac{n-1}{2}} (\p_r - ik) u^s = 0.
\end{equation}
These equations define direct scattering \cite{CK19}, and given $V$ and $u^i$ we can solve for the total field using the Lippmann--Schwinger equation
\begin{equation}
  \label{eq:lippman-schwinger}
  u = u^i - k^2(\Delta + k^2)^{-1}\big( (V-1) u \big).
\end{equation}
Here the Green's operator $(\Delta + k^2)^{-1}$ is defined by
\begin{equation}
  \label{eq:green}
  (\Delta + k^2)^{-1} f(x) = -\frac{i}{4}\left( \frac{k}{2 \pi} \right)^{\frac{n-2}{2}} \int_{R^n} \lvert x-y \rvert^{\frac{2-n}{n}} H^{(1)}_{\frac{n-2}{2}}(k \lvert x-y \rvert) f(y) dy.
\end{equation}
Asymptotic analysis of this operator and the Lippmann--Schwinger equation shows that
\begin{equation}
  \label{eq:u-series}
  u(x) = u^i(x) + \frac{e^{i k \lvert x \rvert}}{\lvert x \rvert^{(n-1)/2}} u_{\infty}(\hat x) + \mathcal O(\lvert x \rvert^{n/2}) \qquad \text{ as } r \to \infty
\end{equation}
where $\hat x = x / r$ is the radial direction and $u_{\infty}^s$ is the far-field pattern of the scattered wave.

The inverse scattering problem is as follows: given the mapping $(u^i, k) \mapsto u_{\infty}^s$, what can we say about $V$? This problem has many variations, e.g. fixed wavenumber $k$ or all wavenumbers, how many and what types of incident waves $u^i$ are available, in which directions can we observe $u_{\infty}^s$, etc. Classical methods for solving inverse scattering problems such as the linear sampling methods of Colton and Kirsch \cite{CKir96} or the factorization method by Kirsch \cite{KigG08} fail when the \emph{far-field operator} has a non-trivial kernel. This means that there is a non-trivial incident wave that produces a trivial far-field pattern. The wavenumber $k$ is called a \emph{non-scattering energy} in these cases.
It also implies that $k$ is an ITEV and that $v = u^i$, $w = u$ solve \eqref{eq_ite}. The converse is not true, as seen by studying scattering from corners \cite{BPS14}. In other words, a non-scattering energy is always an interior transmission eigenvalue, but the converse is not true in general.

\medskip
In this paper we studied the corner behaviour of solutions to the ITE problem \eqref{eq_ite} when the angle is greater or less than $\pi$. In the convex case, we showed that $v=w=0$ in the integral sense at the corner assuming we know a-priori that $v,w \in H^1(D)$. By reducing the scattering problem to the ITE problem, we get yet another proof of the following theorem originally shown\footnote{The result in \cite{PSV17} is more general. By considering all orders of vanishing of $u^i$ they conclude that $u_{\infty}^s=0$ is impossible with a non-trivial incident wave.} in \cite{PSV17}. It can also be proven with alternative methods, e.g. using the result for sources \cite{Bla18} and the connection between source- and potential scattering \cite{BL21}.
\begin{theorem}
  Let $n=2$ and assume that $V$ and $u^i$ produce a trivial far-field $u_{\infty}^s=0$ at wavenumber $k$. If $V_{|D} \in C^1(\overline{D})$ and $V \neq 1$ at a convex corner $x_0$ of $D$ that can be connected to infinity by a chain of balls in $\R^n \setminus \overline{D}$, then one must have $u^i(x_0)=0$ at the corner point.
\end{theorem}
\noindent These types of convex corner results are recent and well-known, but it's interesting that it's possible to finally prove them through the use of inherent properties interior transmission eigenfunctions.

\medskip
A more interesting situation arises when looking at non-convex corners. We proved that there are $L^2(D)$ interior transmission eigenfunctions which have the asymptotic blow-up rate $r^{-\alpha}$, $\alpha \in (0,1)$ as we approach the corner in $D$. Looking at the scattering problem and its usual reduction to the ITE problem \eqref{eq_ite} by setting $v = u^i$ and $w = u$, we see that these new singular ITEFs cannot be obtained through scattering. This is because elliptic regularity guarantees that $u^i, u$ are $H^2$-smooth, which in two dimensions implies boundedness. It is well-known that any $H^1$ ITEF can be approximated by incident Herglotz waves in the $H^1$-norm \cite{Weck03}, however our result says that one cannot go further: \emph{some transmission eigenfunction cannot be incident Herglotz waves or total waves}.

A natural question to ask is: ``what about scattering of non-convex corners?''
Our result shows the existence of some wavenumber $k$ and an associated transmission eigenfunction that's singular. The singularity prevents extension of the eigenfunction to the whole $\R^2$, something noted for convex corners in previous work \cite{BPS14}. However, since it is unclear if every eigenfunction of a non-convex corner is singular, our current result cannot be used to prove something like ``non-convex corners always scatter''.
Looking at the literature, \cite{EH15,EH18} showed this using other methods. On the other hand, the methods of
\cite{BPS14,PSV17,HSV16} and later rely on complex geometrical optics solutions
which decay exponentially only in a half-space, so are not suitable to
non-convex corners. A type of CGO solution that decays in almost all directions
was introduced in \cite{Bla18}, and one could try to answer the question by
reducing the potential scattering problem to source scattering for $u^s$. If
$u^i$ is such that $u_{\infty}^s=0$, then by Rellich's theorem and unique
continuation through a chain of balls prove that $u^s=0$ on the external
boundary of $D$. The scattered wave satisfies then
\[
  (\Delta + k^2)u^s = k^2(1-V) u \qquad \text{ in } \R^n
\]
where $f := k^2 (1-v) u \in H^2(\R^n)$ is H\"older-continuous in two dimensions. Then \cite{Bla18} implies that $f = 0$ in the corners, including the non-convex ones. Since $V \neq 1$ at corners, the total wave must vanish, $u = 0$. Current methods cannot take the next step and show that the incident wave $u^i$ would vanish there without assuming that $k$ is small, or prove that there is higher order vanishing.

The situation of non-convex corners further emphasizes that \emph{$L^2$
interior transmission eigenfunctions is a much more exotic class of functions
than incident waves or total waves}; it even includes unbounded functions in
two dimensions. 

\subsection{Open questions}
The result of Theorem \ref{thm_main_loc} is in many ways incomplete, 
and there are several obvious open open problems related to these phenomena.
Let us mention a few of these.

\begin{itemize}

\item
One of the most immediate open questions is if one can derive our existence results
for non-constant refractive  indices $n$ in \eqref{eq_ite}?\\

\item
Another open questions is whether an ITEF can vanish at a non-convex corner?
And if so, when does this happen? Note that \cite{EH15,EH18} proves that non-convex corners always scatter. This means that ITEFs in this geometry never come from scattering solutions, so almost nothing is known about them.

\item
A related problem is if the singularity can be oscillating? And is there a relation
between the type of singularity and the angle $\omega$, like suggested first in \cite{BLLW17}?

\end{itemize}

\section*{Acknowledgements}

\noindent
V.~P. was supported by the Research Council of Finland (Flagship of Advanced Mathematics for Sensing, Imaging and Modelling grant 359186)
and by the Emil Aaltonen Foundation. E.~B. was supported by the Research Council of Finland (Flagship of Advanced Mathematics for Sensing, Imaging and Modelling grant 359183) and by the Finnish Ministry of Education and Culture's Pilot for Doctoral Programmes (Pilot project Mathematics of Sensing, Imaging and Modelling).

\printbibliography

\end{document}